\ifdefined\useejpstyle
\documentclass[EJP]{ejpecp}

\VOLUME{0}
\YEAR{2013}
\PAPERNUM{0}
\DOI{vVOL-PID}
\else
\documentclass[preprint,12pt]{imsart}
\usepackage[margin=0.9in]{geometry}
\fi

\usepackage{amsthm}
\usepackage{amsmath}
\usepackage{amssymb}

\usepackage{wasysym}
\usepackage{bbm}
\usepackage{bm}
\usepackage{mathrsfs}

\usepackage{color}

\usepackage{hyperref}

\usepackage{xspace}

\newcommand{\Holder}{H\"{o}lder\xspace}

\newcommand{\metricspace}{\mathcal{Z}}

\newtheorem{thm}{Theorem}
\newtheorem{cor}[thm]{Corollary}
\ifdefined\useejpstyle
\else
\newtheorem{lemma}[thm]{Lemma}
\newtheorem{conjecture}[thm]{Conjecture}
\fi
\newtheorem{prop}[thm]{Proposition}

\theoremstyle{definition}

\newtheorem{defn}[thm]{Definition}

\theoremstyle{remark}

\newtheorem{rem}[thm]{Remark}

\numberwithin{equation}{section}

\numberwithin{thm}{section}

\newcommand{\cororef}[1]{Corollary~{\ref{cor:#1}}}

\newcommand{\lemref}[1]{Lemma~{\ref{lem:#1}}}
\newcommand{\defref}[1]{Definition~{\ref{def:#1}}}
\newcommand{\secref}[1]{Section~{\ref{sec:#1}}}
\newcommand{\secsref}[1]{Sections~{\ref{sec:#1}}}
\newcommand{\secssref}[1]{{\ref{sec:#1}}}

\newcommand{\thmref}[1]{Theorem~{\ref{thm:#1}}}

\DeclareMathAlphabet{\mathsfsl}{OT1}{cmss}{m}{sl}

\newcommand{\qtext}[1]{\quad\text{#1}\quad}

\newcommand{\term}{\emph}

\renewcommand{\phi}{\varphi}

\newcommand{\II}{\mathbbm{1}}

\newcommand{\defby}{\mathrel{\mathop:}=}

\newcommand{\half}{\tfrac{1}{2}}

\newcommand{\econst}{\mathrm{e}}

\newcommand{\onevct}{\mathbf{e}}

\newcommand{\Id}{\mathbf{I}}

\newcommand{\zeromtx}{\bm{0}}

\newcommand{\R}{\mathbb{R}}
\newcommand{\C}{\mathbb{C}}

\newcommand{\M}{\mathbb{M}}
\newcommand{\Sym}[1]{\mathbb{H}^{#1}}

\newcommand{\abs}[1]{\left\vert {#1} \right\vert}

\newcommand{\real}{\operatorname{Re}}

\newcommand{\diff}[1]{\mathrm{d}{#1}}
\newcommand{\idiff}[1]{\, \diff{#1}}
\newcommand{\ddt}[1]{\frac{\mathrm{d}}{\mathrm{d}{#1}}}

\newcommand{\Prob}[1]{\mathbb{P}\left\{ {#1} \right\}}
\newcommand{\Expect}{\operatorname{\mathbb{E}}}
\newcommand{\Var}{\operatorname{Var}}

\newcommand{\D}{\mathcal{D}}

\newcommand{\condl}{\, \vert \,}

\newcommand{\vct}[1]{\bm{#1}}
\newcommand{\mtx}[1]{\bm{#1}}

\newcommand{\adj}{*}

\newcommand{\trace}{\operatorname{tr}}
\newcommand{\ntr}{\operatorname{\bar{tr}}}

\def\indep{\perp\!\!\!\perp} 

\newcommand{\psdle}{\preccurlyeq}
\newcommand{\psdge}{\succcurlyeq}

\newcommand{\ip}[2]{\left\langle {#1},\ {#2} \right\rangle}
\newcommand{\absip}[2]{\abs{\ip{#1}{#2}}}

\newcommand{\norm}[1]{\left\Vert {#1} \right\Vert}
\newcommand{\normsq}[1]{\norm{#1}^2}

\newcommand{\smnorm}[2]{{\bigl\Vert {#2} \bigr\Vert}_{#1}}

\newcommand{\indnorm}[2]{\norm{#2}_{#1\to#1}} 
\newcommand{\pnorm}[2]{\norm{#2}_{S_{#1}}} 

\newcommand{\Zvec}{Z}

\newcommand{\event}[0]{\mathcal{E}} %
\newcommand{\tv}[0]{d_{\text{TV}}} %
\def\maxeig#1{\lambda_{{\rm max}}\left({#1}\right)}

\newcommand{\OurAbstract}{
This paper derives exponential tail bounds and polynomial moment inequalities
for the spectral norm deviation of a random matrix from its mean.
The argument depends on a matrix extension of Stein's method of exchangeable pairs
for concentration of measure, as introduced by Chatterjee.
Recent work of Mackey et al.~uses these techniques to analyze
random matrices with additive structure, while the enhancements in this
paper cover a wider class of matrix-valued random elements.  In
particular, these ideas lead to a bounded differences inequality
that applies to random matrices constructed from weakly dependent random variables.
The proofs require novel trace inequalities that may be of independent interest.
}

\ifdefined\useejpstyle
\SHORTTITLE{Matrix Concentration from Kernel Couplings}
\TITLE{Deriving Matrix Concentration Inequalities\\ from Kernel Couplings}

\AUTHORS{%
  Daniel~Paulin\footnote{Department of Mathematics, National University of Singapore. \EMAIL{paulindani@gmail.com}}
  \and 
  Lester~Mackey\footnote{Department of Statistics, Stanford University. \EMAIL{lmackey@stanford.edu}}
  \and
  Joel A. Tropp\footnote{Department of Computing and Mathematical Sciences, California Institute of Technology. \BEMAIL{jtropp@cms.caltech.edu}}
  }

\ABSTRACT{%
\OurAbstract
}

\AMSSUBJ{60B20 ; 60E15}
\AMSSUBJSECONDARY{60G09 ; 60F10}


\SUBMITTED{2 May 2013} %

\begin{document}

\else
\begin{document}

\begin{frontmatter}
\title{Deriving Matrix Concentration Inequalities\\ from Kernel Couplings}
\runtitle{Matrix Concentration from Kernel Couplings}

\begin{aug}
\author{\fnms{Daniel} \snm{Paulin}$^1$\ead[label=e1]{paulindani@gmail.com}}
\\
\author{\fnms{Lester} \snm{Mackey}$^2$\ead[label=e2]{lmackey@stanford.edu}}
\and
\author{\fnms{Joel A.} \snm{Tropp}$^3$\ead[label=e3]{jtropp@cms.caltech.edu}}
\runauthor{Paulin, Mackey, and Tropp}

\affiliation{National University of Singapore, Stanford University, \and California Institute of Technology}

\address{$^1$ Department of Mathematics,
National University of Singapore, \printead{e1}}

\address{$^2$ Department of Statistics, Stanford University, \printead{e2}}

\address{$^3$ Department of Computing and Mathematical Sciences, California Institute of Technology, \\\printead{e3}}
\end{aug}

\begin{abstract}
This paper derives exponential tail bounds and polynomial moment inequalities
for the spectral norm deviation of a random matrix from its mean value.
The argument depends on a matrix extension of Stein's method of exchangeable pairs
for concentration of measure, as introduced by Chatterjee.
Recent work of Mackey et al.~uses these techniques to analyze
random matrices with additive structure, while the enhancements in this
paper cover a wider class of matrix-valued random elements.  In
particular, these ideas lead to a bounded differences inequality
that applies to random matrices constructed from weakly dependent random variables.
The proofs require novel trace inequalities that may be of independent interest.
\end{abstract}


\end{frontmatter}

\fi

\emph{This paper is based on two independent manuscripts from late 2012 that both used kernel couplings to establish matrix concentration inequalities.  One manuscript is by Paulin; the other is by Mackey and Tropp.  The authors have combined this research into a unified presentation, with equal contributions from both groups.}

\vspace{1pc}

\section{Introduction} \label{sec:intro}

Matrix concentration inequalities provide probabilistic bounds on the
spectral-norm deviation of a random matrix from its mean value.
Over the last decade, a growing field of research has established
that many scalar concentration results have direct analogs for matrices.
For example, see~\cite{AW02:Strong-Converse,
Oli10:Concentration-Adjacency,Tro11:User-Friendly-FOCM}.
This machinery has simplified the study of random matrices that arise in
applications from statistics~\cite{koltchinskii2012neumann},
machine learning~\cite{MorvantKoRa12},
signal processing~\cite{ARR12:Blind-Deconvolution},
numerical analysis~\cite{TroppHadamard},
theoretical computer science~\cite{WX08:Derandomizing-Ahlswede-Winter},
and combinatorics~\cite{Oli11:Spectrum-Random}.

Most of the recent research on matrix concentration depends
on a matrix extension of the Laplace transform method from
elementary probability.
In the matrix setting, it is a serious technical challenge to obtain
bounds on the matrix analog of the moment generating function.
The earlier works~\cite{AW02:Strong-Converse,Oli10:Concentration-Adjacency}
use the Golden--Thompson inequality to accomplish this task.
A more powerful argument~\cite{Tro11:User-Friendly-FOCM} invokes
Lieb's Theorem~\cite[Thm.~6]{Lie73:Convex-Trace} to complete
the estimates.

Very recently, Mackey et al.~\cite{MackeyJoChFaTr12} have shown that
it is also possible to use Stein's method of exchangeable pairs to
control the matrix moment generating function.  This argument depends on a
matrix version of Chatterjee's technique~\cite{Cha07:Steins-Method,
Cha08:Concentration-Inequalities} for establishing concentration
inequalities using exchangeable pairs.  This approach has two chief
advantages.  First, it offers a straightforward way to prove polynomial
moment inequalities for matrices, which are not easy to obtain using
earlier techniques.  Second, exchangeable pair arguments also apply
to random matrices constructed from weakly dependent random variables.

The work~\cite{MackeyJoChFaTr12} focuses on sums of weakly
dependent random matrices because its techniques are less effective
for other examples.  The goal of the current research is to adapt ideas
from Chatterjee's thesis~\cite{Cha08:Concentration-Inequalities}
to establish concentration inequalities for more general
types of random matrices.
In particular, we have obtained new versions of the matrix
bounded difference inequality~(see~\cite[Cor.~7.5]{Tro11:User-Friendly-FOCM}
or~\cite[Cor.~11.1]{MackeyJoChFaTr12}) that hold for a random
matrix that is expressed as a measurable function of weakly
dependent random variables.  These results appear as
Corollary~\ref{cor:bound-diff} and Corollary~\ref{cor:dob-bound-diff}.

\subsection{A First Look at Exchangeable Pairs}

The method of exchangeable pairs depends on the idea that an
exchangeable counterpart of a random variable encodes information
about the symmetries in the distribution.
Here is a simple but fundamental example of an exchangeable pair
of random matrices:
\begin{equation} %
\mtx{X} = \sum\nolimits_{j=1}^n \mtx{Y}_j
\quad\text{and}\quad
\mtx{X}' = \mtx{X} + (\widetilde{\mtx{Y}}_J - \mtx{Y}_J)
\end{equation}
where $\{ \mtx{Y}_j \}$ is an independent family of random Hermitian
matrices, $J$ is a random index chosen uniformly from $\{1, \dots, n\}$,
and $\widetilde{\mtx{Y}}_J$ is an independent copy of $\mtx{Y}_J$.
Notice that
$$
\frac{n}{2} \Expect\big[ (\mtx{X} - \mtx{X}')^2 \big]
	= \Expect\big[ \mtx{X}^2 \big] - [\Expect \mtx{X}]^2
	= \Var(\mtx{X}).
$$
As a consequence, we can interpret the random matrix
$\tfrac{n}{2} (\mtx{X} - \mtx{X}')^2$ as a stochastic
estimate for the variance of the independent sum $\mtx{X}$.
When this random matrix is uniformly small in norm,
we can prove that the sum $\mtx{X}$ concentrates
around its mean value.
We refer to Theorem~\ref{thm:concentration-bdd} or the
result~\cite[Thm.~4.1]{MackeyJoChFaTr12} for a
rigorous statement.

\subsection{Roadmap}

Section~\ref{subsec:NotPre}  continues with some notation and preliminary remarks.
In Section~\ref{sec:exchange}, we describe the concept of a kernel Stein pair of random matrices, which stands at the center of our analysis.  In Section~\ref{sec:concentration-bdd}, we state abstract concentration inequalities for kernel Stein pairs. Afterward, Sections~\ref{sec:matrix-bdd-diff} and~\ref{sec:dobrushin} derive bounded difference inequalities for random matrices constructed from independent and weakly dependent random variables.  As an application, we consider the problem of estimating the correlations in a two-dimensional Ising model in Section~\ref{sec:ising}.  We close with some complementary material in Section~\ref{sec:beyond}.  The proofs of the main results appear in three Appendices.

\subsection{Notation and Preliminaries}\label{subsec:NotPre}

First, we introduce the identity matrix $\Id$ and the zero matrix $\zeromtx$.  Their dimensions are determined by context.

We write $\M^d$ for the algebra of $d \times d$ complex matrices.  The symbol $\norm{ \cdot }$ always refers to the usual operator norm on $\M^d$ induced by the $\ell_2^d$ vector norm.  We also equip $\M^d$ with the trace inner product $\ip{ \mtx{B} }{ \mtx{C} } \defby \trace[\mtx{B}^* \mtx{C}]$ to form a Hilbert space.

Let $\Sym{d}$ denote the subspace of $\M^d$ consisting of $d \times d$ Hermitian matrices.  %
Given an interval $I$ of the real line, we define $\Sym{d}(I)$
to be the family of Hermitian matrices with eigenvalues contained in $I$.  We
use curly inequalities, such as $\psdle$, for the positive semidefinite order
on the Hilbert space $\ell_2^d$ and the Hilbert space $\Sym{d}$.

Let $f : I \to \R$ be a function on an interval $I$ of the real line.  We can lift $f$ to form a \term{standard matrix function} $f : \Sym{d}(I) \to \Sym{d}$.  More precisely, for each matrix $\mtx{A} \in \Sym{d}(I)$, we define the standard matrix function via the rule
$$
f(\mtx{A}) := \sum\nolimits_{k=1}^d f(\lambda_k) \, \vct{u}_k \vct{u}_k^\adj
\quad\text{where}\quad
\mtx{A} = \sum\nolimits_{k=1}^d \lambda_k \, \vct{u}_k \vct{u}_k^\adj
$$
is an eigenvalue decomposition of the Hermitian matrix $\mtx{A}$.  When we apply a familiar scalar function to an Hermitian matrix, we are always referring to the associated standard operator function.  To denote general matrix-valued functions, we use bold uppercase letters, such as $\mtx{F}, \mtx{H}, \mtx{\Psi}$.

For $\mtx{M}\in \M^d$, we write $\real(\mtx{M}):= \half (\mtx{M}+\mtx{M}^*)$ for the Hermitian part of $\mtx{M}$. The following semidefinite relation holds.
\begin{equation} \label{eqn:matrix-am-gm}
\real(\mtx{AB}) = \frac{\mtx{AB} + \mtx{BA}}{2}
\psdle \frac{\mtx{A}^2 + \mtx{B}^2}{2}
\quad\text{for all $\mtx{A}, \mtx{B} \in \Sym{d}$.}
\end{equation}
This result follows when we expand the expression $(\mtx{A} - \mtx{B})^2 \psdge \mtx{0}$.  As a consequence,
\begin{equation} \label{eqn:square-convex}
\left(\frac{\mtx{A} + \mtx{B}}{2}\right)^2 \psdle \frac{\mtx{A}^2 + \mtx{B}^2}{2}
\quad\text{for all $\mtx{A}, \mtx{B} \in \Sym{d}$.}
\end{equation}
In other words, the matrix square is operator convex.

Finally, we need two additional families of matrix norms.  For $p \in [1, \infty]$, the Schatten $p$-norm is given by
$$
\pnorm{p}{\mtx{B}} \defby \big( \trace \abs{\mtx{B}}^p \big)^{1/p}
	\quad\text{for each $\mtx{B} \in \M^d$} ,
$$
where $\abs{\mtx{B}} \defby (\mtx{B}^\adj \mtx{B})^{1/2}$.  For $p \geq 1$, we introduce the matrix norm induced by the $\ell_p^d$ vector norm:
\begin{equation} \label{eqn:induced-norm}
\indnorm{p}{\mtx{B}} \defby \sup_{\vct{x} \neq \vct{0}} \ \frac{\norm{\mtx{B}\vct{x}}_p}{\norm{\vct{x}}_p}
\quad\text{for each $\mtx{B} \in \M^d$}
\end{equation}
In particular, the matrix norm induced by the $\ell_1^d$ vector norm returns the maximum $\ell_1^d$ norm of a column; the norm induced by $\ell_\infty^d$ returns the maximum $\ell_1^d$ norm of a row.

\section{Exchangeable Pairs of Random Matrices} \label{sec:exchange}

The basic principle behind this paper is that
we can exploit the symmetries of the distribution of a random
matrix to obtain matrix concentration inequalities.
One way to encode symmetries is to identify an exchangeable counterpart
of the random matrix.  This section
outlines the main concepts from the method of exchangeable
pairs, including an example of fundamental importance.
Once we have an exchangeable pair, we can apply
ideas of Chatterjee~\cite{Cha08:Concentration-Inequalities}
to obtain concentration inequalities,
which is the subject of Section~\ref{sec:concentration-bdd}.

\subsection{Kernel Stein Pairs} \label{sec:pairs}

In this work, the primal concept is an exchangeable pair of random variables.  

\begin{defn}[Exchangeable Pair] \label{def:exchange}
Let $Z$ and $Z'$ be a pair of random variables taking values in a Polish space $\metricspace$.
We say that a $(Z, Z')$ is an \term{exchangeable pair} when it has the
same distribution as the pair $(Z', Z)$.
\end{defn}

\noindent
In particular, $Z$ and $Z'$ have the same distribution,
and $\Expect f(Z, Z') = \Expect f(Z',Z)$
for every function $f$ where the expectations are finite.

We are interested in a special class of exchangeable pairs of random matrices.
There must be an antisymmetric bivariate kernel that ``reproduces'' the matrices
in the pair.

\begin{defn}[Kernel Stein Pair] \label{def:K-stein-pair}
Let $(Z, Z')$ be an exchangeable pair of random variables taking values in a Polish space $\metricspace$,
and let $\mtx{\Psi} : \metricspace \to \Sym{d}$ be a measurable function. 
Define the random Hermitian matrices
$$
\mtx{X} \defby \mtx{\Psi}(Z)
\quad\text{and}\quad
\mtx{X}' \defby \mtx{\Psi}(Z').
$$
We say that $(\mtx{X}, \mtx{X}')$ is a \term{kernel Stein pair} if there is a bivariate function $\mtx{K} : \metricspace^2 \to \Sym{d}$ for which
\begin{equation} \label{eqn:K-stein-pair}
\mtx{K}(Z,Z') = -\mtx{K}(Z',Z)
~~\text{and}~~
\Expect[ \mtx{K}(Z,Z') \condl Z ] = \mtx{X} 
	\quad\text{almost surely.}
\end{equation}
When discussing a kernel Stein pair $(\mtx{X}, \mtx{X}')$, we always assume that $\Expect \normsq{\mtx{X}} < \infty$.  We sometimes write \term{$\mtx{K}$-Stein pair} to emphasize the specific kernel $\mtx{K}$.
\end{defn}

\noindent
It turns out that most exchangeable pairs of random matrices admit
a kernel $\mtx{K}$ that satisfies~\eqref{eqn:K-stein-pair}.  We
describe the construction in Section~\ref{sec:couplings}.

\paragraph{Kernel Stein Pairs versus Matrix Stein Pairs.}

The analysis in the article~\cite{MackeyJoChFaTr12} is based on an important subclass of kernel Stein pairs termed \term{matrix Stein pairs}.
A matrix Stein pair $(\mtx{X}, \mtx{X}')$ derived from an auxiliary exchangeable pair $(Z,Z')$ satisfies the stronger condition
\begin{equation} \label{eqn:matrix-stein-pair}
\Expect[ \mtx{X} - \mtx{X}' \condl Z ] = \alpha \mtx{X}
\quad\text{for some $\alpha > 0$.}
\end{equation}
That is, a matrix Stein pair is a kernel Stein pair with $\mtx{K}(Z,Z') = \alpha^{-1} (\mtx{X} - \mtx{X}' )$.  Although the paper~\cite{MackeyJoChFaTr12} describes several fundamental classes of matrix Stein pairs, most exchangeable pairs of random matrices do not satisfy the condition~\eqref{eqn:matrix-stein-pair}.  Kernel Stein pairs are much more common, so they are commensurately more useful.

\subsection{Kernel Couplings} \label{sec:couplings}

Given an exchangeable pair of random matrices, we can ask whether it is
possible to equip the pair with a kernel that satisfies~\eqref{eqn:K-stein-pair}.
In fact, there is a very general construction that works whenever the
exchangeable pair is suitably ergodic.  This method depends on
an idea of Chatterjee~\cite[Sec.~4.1]{Cha08:Concentration-Inequalities}
that ultimately relies on an observation of Stein~\cite{Stein86}.

Stein noticed that any exchangeable pair $(Z,Z')$
of $\metricspace$-valued random variables defines a reversible Markov chain
with a symmetric transition kernel $P$ given by
$$
Pf(z) \defby \Expect[ f(Z') \condl Z = z]
$$
for each integrable function $f : \metricspace \to \R$.  In other words,
for any initial value $Z_{(0)} \in \metricspace$, we can construct a
Markov chain
$$
Z_{(0)} \to Z_{(1)} \to Z_{(2)} \to Z_{(3)} \to \cdots
$$
where $\Expect[ f(Z_{(i+1)}) \condl Z_{(i)} ] = Pf(Z_{(i)})$
for each integrable function $f$.  This requirement suffices
to determine the distribution of each $Z_{(i+1)}$.

When the chain $(Z_{(i)})_{i \geq 0}$ is ergodic enough,
we can explicitly construct a kernel that
satisfies~\eqref{eqn:K-stein-pair} for any exchangeable
pair of random matrices constructed from the auxiliary
exchangeable pair $(Z,Z')$.  To explain this idea, we
begin with a definition.

\begin{defn}[Kernel Coupling] \label{def:kernel-coupling}
Let $(Z,Z') \in \metricspace^2$ be an exchangeable pair. 
Let $(Z_{(i)})_{i\geq 0}$ and $(Z'_{(i)})_{i\geq 0}$
be two Markov chains with arbitrary initial values, each
evolving according to the transition kernel $P$ induced
by $(Z,Z')$.
We call $(Z_{(i)},Z'_{(i)})_{i\geq 0}$ a \emph{kernel coupling} for $(Z,Z')$ if, 
\begin{align} \label{eqn:kernel-coupling}
Z_{(i)} \indep Z_{(0)}' \condl Z_{(0)}\quad\text{and}\quad Z'_{(i)} \indep Z_{(0)} \condl Z'_{(0)} \quad \text{for all $i$.}
\end{align}
The expression $U \indep V \condl W$ means that $U$ and $V$ are independent conditional on $W$.
\end{defn}

The key lemma, essentially due to Chatterjee~\cite[Sec.~4.1]{Cha08:Concentration-Inequalities}, allows us to construct a kernel Stein pair by way of a kernel coupling.

\begin{lemma} \label{lem:kernel-coupling}
	Let $(Z_{(i)},Z'_{(i)})_{i\geq 0}$ be a kernel coupling for an exchangeable pair $(Z,Z') \in \metricspace^2$.
	Let $\mtx{\Psi} : \metricspace \to \Sym{d}$ be a measurable function with $\Expect\mtx{\Psi}(Z) = \zeromtx.$
	Suppose that there is a positive constant $L$ for which
	\begin{equation}\label{eqn:coupling-premise}
	\sum\nolimits_{i=0}^\infty \norm{ \Expect[\mtx{\Psi}(Z_{(i)}) - \mtx{\Psi}(Z'_{(i)})\condl Z_{(0)}=z,Z'_{(0)}=z']} \leq L \quad\text{for all $z, z' \in \metricspace$}.
	\end{equation}
	Then $(\mtx{\Psi}(Z), \mtx{\Psi}(Z'))$ is a kernel Stein pair
	with kernel 
	\begin{align} \label{eqn:K-construct}
		\mtx{K}(Z,Z') \defby \sum\nolimits_{i=0}^\infty \Expect[\mtx{\Psi}(Z_{(i)}) - \mtx{\Psi}(Z'_{(i)})\condl Z_{(0)}=Z,Z'_{(0)}=Z'].
	\end{align}
\end{lemma}

\noindent
The proof of this result is identical with that of \cite[Lem.~4.2]{Cha08:Concentration-Inequalities}, which establishes the same formula~\eqref{eqn:K-construct} in the scalar setting.  Lemma~\ref{lem:kernel-coupling} indicates that the kernel $\mtx{K}$ associated with an exchangeable pair $(Z, Z')$ and a map $\mtx{\Psi}$ tends to be small when
the two Markov chains in the kernel coupling have a small coupling time.

\subsection{Conditional Variance}

To each kernel Stein pair $(\mtx{X}, \mtx{X}')$, we may associate two random matrices called the \term{conditional variance}  and \term{kernel conditional variance} of $\mtx{X}$.  Ultimately, we show that $\mtx{X}$ is concentrated around the zero matrix whenever the conditional variance and the kernel conditional variance are both small.

\begin{defn}[Conditional Variance] \label{def:conditional-variance}
Suppose that $(\mtx{X}, \mtx{X}')$ is a $\mtx{K}$-Stein pair, 
constructed from an auxiliary exchangeable pair $(Z, Z')$.
The \term{conditional variance} is the random matrix
\begin{align} \label{eqn:conditional-variance}
\mtx{V}_{\mtx{X}}
	\defby \mtx{V}_{\mtx{X}}(Z)
	\defby  \frac{1}{2} \Expect \big[(\mtx{X} - \mtx{X}')^2\condl Z \big],
\end{align}
and the \term{kernel conditional variance} is the random matrix
\begin{align} \label{eqn:K-conditional-variance}
\mtx{V}^{\mtx{K}}
	\defby \mtx{V}^{\mtx{K}}(Z)
	\defby  \frac{1}{2} \Expect \big[\mtx{K}(Z,Z')^2\condl Z \big].
\end{align}
\end{defn}

The following lemma provides a convenient way to control the conditional variance and the kernel conditional variance when the kernel is obtained from a kernel coupling as in Lemma~\ref{lem:kernel-coupling}.

\begin{lemma} \label{lem:conditional-variance-bound}
Let $(Z_{(i)},Z'_{(i)})_{i\geq 0}$ be a kernel coupling for an exchangeable pair $(Z,Z') \in \metricspace^2$, and let $\mtx{\Psi} : \metricspace \to \Sym{d}$ be a measurable map.  Suppose that $(\mtx{X}, \mtx{X}') = (\mtx{\Psi}(Z), \mtx{\Psi}(Z'))$ is a kernel Stein pair where the kernel $\mtx{K}$ is constructed via \eqref{eqn:K-construct}.  
For each $i = 0, 1, 2, \dots$, assume that
\begin{align} \label{eqn:K-term-bound}
\Expect\big[\Expect[\mtx{\Psi}(Z_{(i)}) - \mtx{\Psi}(Z'_{(i)})\condl Z,Z']^2\condl Z \big] 
	\psdle s_i^2 \, \mtx{\Gamma}(Z)\quad\text{almost surely},
\end{align}
where $\mtx{\Gamma} : \metricspace \to \Sym{d}$ is a measurable map and $(s_i)_{i\geq 0}$ is a deterministic sequence of nonnegative numbers.
Then the conditional variance \eqref{eqn:conditional-variance} satisfies
\begin{align*} 
\mtx{V}_{\mtx{X}} 
	\psdle \frac{1}{2} s_0^2 \, \mtx{\Gamma}(Z)\quad\text{almost surely},
\end{align*}
and the kernel conditional variance \eqref{eqn:K-conditional-variance} satisfies
\begin{align*} 
\mtx{V}^{\mtx{K}} 
	\psdle \frac{1}{2}\left(\sum\nolimits_{i=0}^\infty s_i\right)^2 \mtx{\Gamma}(Z)\quad\text{almost surely}.
\end{align*}
\end{lemma}

\begin{proof}
Using a continuity argument, we may assume that each $s_i > 0$ for each integer $i \geq 0$.  For each $i$, define $\mtx{Y}_i \defby \Expect[\mtx{\Psi}(Z_{(i)}) - \mtx{\Psi}(Z'_{(i)})\condl Z,Z']$.
By the kernel coupling construction \eqref{eqn:K-construct}, we have
\begin{align*}
\mtx{V}^{\mtx{K}} 
	&= \frac{1}{2}\sum\nolimits_{i=0}^\infty \sum\nolimits_{j=0}^\infty 
	\Expect[\mtx{Y}_i\mtx{Y}_j \condl Z] 
	= \frac{1}{2}\sum\nolimits_{i=0}^\infty \sum\nolimits_{j=0}^\infty 
	\Expect[\real(\mtx{Y}_i\mtx{Y}_j) \condl Z]  \\
	&\psdle \frac{1}{2}\sum\nolimits_{i=0}^\infty \sum\nolimits_{j=0}^\infty 
	\frac{1}{2}\left(\frac{s_j}{s_i}\Expect[\mtx{Y}_i^2\condl Z]+ \frac{s_i}{s_j}\Expect[\mtx{Y}_j^2\condl Z]\right) \\
	&\psdle \frac{1}{2}\sum\nolimits_{i=0}^\infty \sum\nolimits_{j=0}^\infty 
	\frac{1}{2}\left(\frac{s_j}{s_i}s_i^2 \, \mtx{\Gamma}(Z)+ \frac{s_i}{s_j}s_j^2 \,\mtx{\Gamma}(Z)\right) \\
	&=\frac{1}{2} \left( \sum\nolimits_{i=0}^\infty s_i \sum\nolimits_{j=0}^\infty s_j \right) \mtx{\Gamma}(Z)  
	= \frac{1}{2}\left(\sum\nolimits_{i=0}^\infty s_i\right)^2 \mtx{\Gamma}(Z) ,
\end{align*}
where the first semidefinite inequality follows from~\eqref{eqn:matrix-am-gm} and the second inequality depends on the hypothesis \eqref{eqn:K-term-bound}.
Similarly,
\begin{align*}
\mtx{V}_{\mtx{X}} 
	= \frac{1}{2}
	\Expect[\mtx{Y}_0^2 \condl Z] 
	\psdle\frac{1}{2} s_0^2 \,  \mtx{\Gamma}(Z).
\end{align*}
This observation completes the proof.
\end{proof}

\subsection{Example: Matrix Functions of Independent Variables} \label{sec:bounded-diffs}

To illustrate the definitions in this section, we describe a simple but important example of a kernel Stein pair.
Suppose that $\Zvec \defby (Z_1, \dots, Z_n)$ is a vector of independent random variables taking values in a Polish space $\metricspace$.
Let $\mtx{H} : \metricspace \to \Sym{d}$ be a measurable function, and let $(\mtx{A}_j)_{j\geq 1}$ be a sequence of deterministic Hermitian matrices satisfying
\begin{equation} \label{eqn:bdd-diff-hyp}
(\mtx{H}(z_1,\dots, z_n) - \mtx{H}(z_1,\dots,z_j',\dots,z_n))^2 \psdle \mtx{A}_j^2
\end{equation}
where $z_j, z_j'$ range over the possible values of $Z_j$ for each $j$.
We aim to analyze the random matrix %
\begin{equation} \label{eqn:bdd-diff-mtx}
\mtx{X} \defby \mtx{H}(\Zvec) - \Expect \mtx{H}(\Zvec).
\end{equation}
We encounter matrices of this form in a variety of applications.  For instance, concentration inequalities for the norm of $\mtx{X}$ have immediate implications for the generalization properties of algorithms for multiclass classification~\cite{MachartRa12,MorvantKoRa12}.

In this section, we explain how to construct a kernel exchangeable pair for studying the random matrix~\eqref{eqn:bdd-diff-mtx}, and we compute the conditional variance and kernel conditional variance.  Later, in Section~\ref{sec:matrix-bdd-diff}, we use these calculations to establish a matrix bounded difference inequality that improves on~\cite[Cor~7.5]{Tro11:User-Friendly-FOCM}.

To begin, we form an exchangeable counterpart for $\Zvec$:
\begin{equation*} %
\Zvec' \defby (Z_1, \dots, Z_{J-1}, \tilde{Z}_{J}, Z_{J+1}, \dots, Z_n)
\end{equation*}
where $\tilde{\Zvec} := (\tilde{Z}_1,\dots, \tilde{Z}_n)$ is an independent copy of $\Zvec$.  We draw the coordinate $J$ uniformly at random from $\{1, \dots, n\}$, independent from everything else.  Then the random matrix
$$
\mtx{X}' \defby \mtx{H}(\Zvec') - \Expect \mtx{H}(\Zvec)
$$
is an exchangeable counterpart for the matrix $\mtx{X}$. 

To verify that $(\mtx{X}, \mtx{X}')$ is a kernel Stein pair for a suitable kernel $\mtx{K}$, we establish an explicit kernel coupling
$(\Zvec_{(i)}, \Zvec_{(i)}')_{i\geq 0}$.
For each $i \geq 1$, define $\tilde{\Zvec}_{(i)}$ to be an independent copy of $\Zvec$. We generate the pair $(\Zvec_{(i)},\Zvec'_{(i)})$ from the previous pair $(\Zvec_{(i-1)},\Zvec'_{(i-1)})$ by selecting an independent random index $J_i$ uniformly from $\{1,\dots,n\}$ and replacing the $J_i$-th coordinates of both $\Zvec_{(i-1)}$ and $\Zvec'_{(i-1)}$ with the $J_i$-th coordinate of $\tilde{\Zvec}_{(i)}$.
By construction, the two marginal chains $(\Zvec_{(i)})_{i\geq 0}$ and $(\Zvec'_{(i)})_{i\geq 0}$ evolve according to the transition kernel induced by $(\Zvec,\Zvec')$, and they satisfy the kernel coupling property \eqref{eqn:kernel-coupling}.
The analysis of the coupon collector's problem~\cite[Sec.~2.2]{LPW10:Markov-Chains} shows that the expected coupling time for this pair of Markov chains is bounded by $n (1 + \log n)$.  Therefore, Lemma~\ref{lem:kernel-coupling} implies that $(\mtx{X}, \mtx{X}')$ is a kernel Stein pair with 
$$
\mtx{K}(\Zvec,\Zvec') \defby \sum\nolimits_{i=0}^\infty \Expect[\mtx{H}(\Zvec_{(i)}) - \mtx{H}(\Zvec'_{(i)})\condl \Zvec_{(0)}=\Zvec,\Zvec'_{(0)}=\Zvec'].
$$
Since the two Markov chains couple rapidly, we expect that the kernel is small.

To bound the size of the kernel, we use Lemma~\ref{lem:conditional-variance-bound}.
For each integer $i \geq 0$, define the event $\event_i \defby \{J\notin \{J_1,\ldots,J_i\}\}$.
Off of the event $\event_i$, we have $\mtx{H}(\Zvec_{(i)}) = \mtx{H}(\Zvec'_{(i)})$; on the event $\event_i$, the random vectors $\Zvec_{(i)}$ and $\Zvec'_{(i)}$ can differ only in the $J$-th coordinate.  Therefore,
\begin{align*}
\Expect\big[\Expect[\mtx{H}(\Zvec_{(i)}) &- \mtx{H}(\Zvec'_{(i)})\condl \Zvec,\Zvec']^2\condl \Zvec \big] \\
&= \Expect\big[(\Prob{\event_{i}}\cdot \Expect[\mtx{H}(\Zvec_{(i)}) - \mtx{H}(\Zvec'_{(i)})\condl \Zvec,\Zvec',\event_{i}])^2\condl \Zvec \big] \\
&\psdle (1-1/n)^{2i} \cdot \Expect[(\mtx{H}(\Zvec_{(i)}) - \mtx{H}(\Zvec'_{(i)}))^2\condl \Zvec,\event_{i}] \\
&\psdle (1-1/n)^{2i} \cdot \Expect[ \mtx{A}_J^2].
\end{align*}
The first semidefinite inequality follows from the convexity~\eqref{eqn:square-convex} of the matrix square, and the second depends on our bounded differences assumption~\eqref{eqn:bdd-diff-hyp}.  Apply Lemma~\ref{lem:conditional-variance-bound} with $s_i = (1-1/n)^{i}$ and $\mtx{\Gamma}(\Zvec) = \Expect[ \mtx{A}_J^2]$ to conclude that
\begin{equation} \label{eqn:bdd-diff-DeltaXK}
\mtx{V}^{\mtx{K}} 
	\psdle  \frac{1}{2}\Expect[ \mtx{A}_J^2] \left(\sum\nolimits_{i=0}^\infty  (1-1/n)^{i}\right)^2
	= \frac{n^2}{2}\Expect[ \mtx{A}_J^2] = \frac{n}{2}\sum\nolimits_{j=1}^n \mtx{A}_j^2
\end{equation}
and that
\begin{align} \label{eqn:bdd-diff-DeltaX} 
\mtx{V}_{\mtx{X}} 
	&\psdle \frac{1}{2}\Expect[ \mtx{A}_J^2] = \frac{1}{2n}\sum\nolimits_{j=1}^n \mtx{A}_j^2.
\end{align}
We discover that the conditional variance and the kernel conditional variance are under control when $\mtx{H}$ has bounded coordinate differences.  Section~\ref{sec:matrix-bdd-diff} discusses how these estimates imply that the matrix $\mtx{X}$ concentrates well.

\section{Concentration Inequalities for Random Matrices} \label{sec:concentration-bdd}

This section contains our main results on concentration for random matrices.  Given a kernel Stein pair, we explain how the conditional variance and kernel conditional variance allow us to obtain exponential tail bounds and polynomial moment inequalities.

At a high level, our work suggests the following plan of action.  You begin with a random matrix, $\mtx{X}$.  You use the symmetries of the random matrix to construct an exchangeable counterpart, $\mtx{X}'$, that is close but not identical to $\mtx{X}$. 
You construct a kernel coupling from this exchangeable pair, and you compute the conditional variances, $\mtx{V}_{\mtx{X}}$ and $\mtx{V}_{\mtx{K}}$.  Then you apply the concentration results from this section to control the deviation of $\mtx{X}$ from its mean.  In the sections to come, we provide specific examples and applications of this template.

\subsection{Exponential Tail Bounds}

Our first result establishes exponential concentration for the maximum and minimum eigenvalues of a random matrix.

\begin{thm}[Concentration for Bounded Random Matrices] \label{thm:concentration-bdd}
Consider a $\mtx{K}$-Stein pair \\$(\mtx{X}, \mtx{X}') \in \Sym{d} \times \Sym{d}$.
Suppose there exist nonnegative constants $c, v, s$ for which the conditional variance~\eqref{eqn:conditional-variance} and the kernel conditional variance~\eqref{eqn:K-conditional-variance} of the pair satisfy
\begin{equation} \label{eqn:comparison}
\mtx{V}_{\mtx{X}} \psdle s^{-1} \cdot (c \mtx{X} + v \, \Id)  
\quad\text{and}\quad
\mtx{V}^{\mtx{K}} \psdle s \cdot (c \mtx{X} + v \, \Id)
\quad \text{almost surely}.
\end{equation}
Then, for all $t \geq 0$,
\begin{align*}
\Prob{ \lambda_{\min}( \mtx{X} ) \leq -t }
	&\leq d \cdot \exp\left\{ \frac{-t^2}{2v} \right\} \\
\Prob{ \lambda_{\max}( \mtx{X} ) \geq t }
	&\leq d \cdot \exp\left\{ -\frac{t}{c} + \frac{v}{c^2} \log\left(1 + \frac{ct}{v} \right) \right\} \\
	&\leq d \cdot \exp\left\{ \frac{-t^2}{2v + 2ct} \right\}.
\end{align*}
Furthermore,
\begin{align*}
\Expect \lambda_{\min}(\mtx{X})
	&\geq -\sqrt{2v\log d}\\
\Expect \lambda_{\max}(\mtx{X})
	&\leq \phantom{-} \sqrt{2v\log d} + c \log d.
\end{align*}
\end{thm}

\noindent
\thmref{concentration-bdd} extends the concentration result of \cite[Thm.~4.1]{MackeyJoChFaTr12}, which only applies to matrix Stein pairs.
The argument leading up to \thmref{concentration-bdd} is very similar with the proof of the earlier result.  The main innovation is a new type of mean value inequality for matrices that improves on~\cite[Lem.~3.4]{MackeyJoChFaTr12}.

\begin{lemma}[Exponential Mean Value Trace Inequality] \label{lem:emvti} 
For all matrices $\mtx{A}, \mtx{B}, \mtx{C} \in \Sym{d}$ and all $s > 0$ it holds that
\begin{align*}
|\trace \left[\mtx{C} (\econst^{\mtx{A}} - \econst^{\mtx{B}} )\right]| \leq  
	\frac{1}{4} \trace[(s \, (\mtx{A}-\mtx{B})^2+ s^{-1} \, \mtx{C}^2)(\econst^{\mtx{A}}+\econst^{\mtx{B}})] .
\end{align*}
\end{lemma}

\noindent
See Appendix~\ref{sec:proof-concentration-bdd} for the proofs of \thmref{concentration-bdd} and \lemref{emvti}.

\subsection{Polynomial Moment Inequalities}

The second main result shows that we can bound the polynomial moments of
a random matrix in terms of the conditional variance and the kernel
conditional variance.

\begin{thm} [Matrix BDG Inequality] \label{thm:bdg-inequality}
Suppose that $(\mtx{X}, \mtx{X}')$ is a $\mtx{K}$-Stein pair based
on an auxiliary exchangeable pair $(Z,Z')$.
Let $p \geq 1$ be a natural number, and
assume that $\Expect \pnorm{2p}{\mtx{X}}^{2p} < \infty$
and $\Expect \norm{\mtx{K}(Z,Z')}^{2p} < \infty$.
Then, for any $s > 0$,
$$
\big( \Expect \pnorm{2p}{\mtx{X}}^{2p} \big)^{1/2p} 
	\leq \sqrt{2p-1} \left( \Expect \pnorm{p}{\frac{1}{2}(s \, \mtx{V}_{\mtx{X}} + s^{-1} \, \mtx{V}^{\mtx{K}}) }^p \right)^{1/2p}.
$$
We have written $\pnorm{p}{\cdot}$ for the Schatten $p$-norm.
\end{thm}

\noindent
\thmref{bdg-inequality} generalizes the matrix Burkholder--Davis--Gundy inequality \cite[Thm.~7.1]{MackeyJoChFaTr12}, which only applies to matrix Stein pairs.
This result depends on another novel mean value inequality for
matrices.

\begin{lemma}[Polynomial Mean Value Trace Inequality] \label{lem:pmvti} 
For all matrices $\mtx{A}, \mtx{B}, \mtx{C} \in \Sym{d}$, all integers $q \geq 1$, and all $s >0$, it holds that
\begin{align*}
\abs{\trace \left[\mtx{C} (\mtx{A}^{q} - \mtx{B}^{q} )\right]} \leq  
	\frac{q}{4} \trace \big[(s \, (\mtx{A}-\mtx{B})^2+s^{-1} \, \mtx{C}^2)(\abs{\mtx{A}}^{q-1} + \abs{\mtx{B}}^{q-1}) \big].
\end{align*}
\end{lemma}

\noindent
The proofs of \thmref{bdg-inequality} and \lemref{pmvti} can be found in Appendix~\ref{sec:proof-bdg-inequality}.  We remark that both results extend directly to infinite-dimensional Schatten-class operators.

\section{Example: Matrix Bounded Differences Inequality} \label{sec:matrix-bdd-diff}

As a first example, we show how to use Theorem~\ref{thm:concentration-bdd} to derive a matrix version of McDiarmid's bounded differences inequality~\cite{McDiarmid89}.  

\begin{cor}[Matrix Bounded Differences] \label{cor:bound-diff}
	Suppose that $\Zvec \defby (Z_1, \dots, Z_n) \in \metricspace$
	is a vector of independent random variables that takes values in a Polish space $\metricspace$.
	Let $\mtx{H} : \metricspace \to \Sym{d}$ be a measurable function,
	and let $(\mtx{A}_1, \dots, \mtx{A}_n)$ be
 	a deterministic sequence of Hermitian matrices that satisfy
	$$
	(\mtx{H}(z_1,\dots, z_n) - \mtx{H}(z_1,\dots,z_j',\dots,z_n))^2 \psdle \mtx{A}_j^2
	$$
	where $z_k,z_k'$ range over the possible values of $Z_k$ for each $k$.
	Compute the boundedness parameter
	$$
	\sigma^2 \defby \norm{ \sum\nolimits_{j=1}^n \mtx{A}_j^2 }.
	$$
	Then, for all $t\geq 0$, 
	$$
	\Prob{ \lambda_{\max}\left( \mtx{H}(\Zvec) - \Expect \mtx{H}(\Zvec) \right) \geq t }
		\leq d \cdot \econst^{-t^2/\sigma^2}.
	$$
	Furthermore,
	$$
	\Expect \lambda_{\max}\left( \mtx{H}(\Zvec) - \Expect \mtx{H}(\Zvec) \right) \leq \sigma \sqrt{\log d}.
	$$ 
\end{cor}

\begin{proof}
Introduce the random matrix $\mtx{X} = \mtx{H}(\Zvec) - \Expect \mtx{H}(\Zvec)$.  
We can use the kernel Stein pair constructed in Section~\ref{sec:bounded-diffs}
to study the behavior of $\mtx{X}$.  According to~\eqref{eqn:bdd-diff-DeltaX}, the conditional variance satisfies
$$
\mtx{V}_{\mtx{X}} 
	\psdle \frac{1}{2n}\sum\nolimits_{j=1}^n \mtx{A}_j^2
	\psdle \left(\frac{\sigma^2}{2} \, \Id\right)/n.
$$
According to~\eqref{eqn:bdd-diff-DeltaXK}, the kernel conditional variance satisfies
$$
\mtx{V}^{\mtx{K}} 
	\psdle \frac{n}{2}\sum\nolimits_{j=1}^n \mtx{A}_j^2
	\psdle n\left(\frac{\sigma^2}{2} \, \Id\right),
$$
Invoke Theorem~\ref{thm:concentration-bdd} with $c = 0$, $v = \sigma^2/2$, and $s = n$
to complete the bound.
\end{proof}

Corollary~\ref{cor:bound-diff} improves on the matrix bounded differences inequality~\cite[Cor.~7.5]{Tro11:User-Friendly-FOCM}, which features an additional factor of 1/8 in the exponent of the tail bound.  It also strengthens the bounded differences inequality~\cite[Cor.~11.1]{MackeyJoChFaTr12} for matrix Stein pairs, which requires an extra assumption that the function $\mtx{H}$ is ``self-reproducing.''

\begin{rem}[Extensions]
The conclusions of Corollary~\ref{cor:bound-diff} hold with $\sigma^2 \defby  \smnorm{}{ \mtx{A}^2 }$ under either one of the weaker hypotheses
$$
\sum\nolimits_j(\mtx{H}(z_1,\dots, z_n) - \mtx{H}(z_1,\dots,z_j',\dots,z_n))^2 \psdle \mtx{A}^2
$$
or
$$
\sum\nolimits_j\Expect\big[(\mtx{H}(z_1,\dots, z_n) - \mtx{H}(z_1,\dots,Z_j\dots,z_n))^2 \big] \psdle \mtx{A}^2
$$
where $\mtx{A} \in \Sym{d}$ is deterministic and $z_k, z_k'$ range over all possible values of $Z_k$ for each index $k$.  This claim follows from a simple adaptation of the argument in Section~\ref{sec:bounded-diffs}.

We can also obtain moment inequalities for the random matrix $\mtx{H}(\vct{Z})$ by invoking Theorem~\ref{thm:bdg-inequality}.  We have omitted a detailed statement because exponential tail bounds are more popular in applications.
\end{rem}

\section{Example: Matrix Bounded Differences without Independence} \label{sec:dobrushin}

A key strength of the method of exchangeable pairs is the fact that it also applies
to random matrices that are built from weakly dependent random variables.  This
section describes an extension of Corollary~\ref{cor:bound-diff} that holds even
when the input variables exhibit some interactions.

To quantify the amount of dependency among the variables, we use
a Dobrushin interdependence matrix~\cite{dobrushin1970prescribing}.
This concept involves a certain amount of auxiliary notation.
Given a vector $\vct{x} = (x_1, \dots, x_n)$,
we write $\vct{x}_{-i} = (x_1, \dots x_{i-1}, x_{i+1}, \dots, x_n)$
for the vector with its $i$th component deleted.
Let $\Zvec = (Z_1, \dots, Z_n)$ be a vector of random variables
taking values in a Polish space $\metricspace$ with sigma algebra
$\mathcal{F}$.  The symbol $\mu_i(\cdot \condl \Zvec_{-i})$ refers to
the distribution of $\Zvec_i$ conditional on the random vector
$\Zvec_{-i}$.  We also require the total variation distance $\tv$
between probability measures $\mu$ and $\nu$ on $(\metricspace, \mathcal{F})$:
\begin{align} \label{eqn:tv}
\tv(\nu,\mu) \defby \sup_{A\in\mathcal{F}} \abs{\nu(A) - \mu(A)}.
\end{align}
With this foundation in place, we can state the definition.

\begin{defn}[Dobrushin Interdependence Matrix] \label{def:dobrushin}
Let $\Zvec=(Z_1,$ $\ldots,Z_n)$ be a random vector taking values in a Polish space $\metricspace$.
Let $\mtx{D} \in \R^{n \times n}$ be a matrix with a zero diagonal that satisfies
the condition
\begin{align} \label{eqn:dobrushin}
\tv\big(\mu_i(\cdot\condl  \vct{x}_{-i}),\mu_i(\cdot\condl  \vct{y}_{-i}) \big)
\leq \sum\nolimits_{j=1}^n D_{ij}\II[x_j\ne y_j]
\end{align}
for each index $i$ and for all vectors $\vct{x}, \vct{y} \in \metricspace$.  Then $\mtx{D}$ is called a \term{Dobrushin interdependence matrix} for the random vector $\Zvec$.
\end{defn}

The kernel coupling method extends readily to the setting of weak dependence.
We obtain a new matrix bounded differences inequality, which is a significant
extension of Corollary~\ref{cor:bound-diff}.  This statement can be viewed as
a matrix version of Chatterjee's result~\cite[Thm.~4.3]{Cha08:Concentration-Inequalities}.

\begin{cor}[Dobrushin Matrix Bounded Differences] \label{cor:dob-bound-diff}
	Suppose that $\Zvec \defby (Z_1, \dots, Z_n)$ in a Polish space $\metricspace$ is a vector of dependent random variables with a Dobrushin interdependence matrix $\mtx{D}$
with the property that
\begin{align}\label{eqn:dobrushin-constraint}
\max\big\{ \indnorm{1}{\mtx{D}}, \ \indnorm{\infty}{\mtx{D}} \big\} < 1.
\end{align} 
Let $\mtx{H} : \metricspace \to \Sym{d}$ be a measurable function,
and let $(\mtx{A}_1, \dots, \mtx{A}_n)$
be a deterministic sequence of Hermitian matrices that satisfy
$$
(\mtx{H}(z_1,\dots, z_n) - \mtx{H}(z_1,\dots,z_j',\dots,z_n))^2 \psdle \mtx{A}_j^2
$$
where $z_k, z_k'$ range over the possible values of $Z_k$ for each $k$. Compute the 
boundedness and dependence parameters
\begin{equation*}
\sigma^2\defby \norm{ \sum\nolimits\nolimits_{j=1}^n \mtx{A}_j^2 }
\qtext{and}
b \defby \left[ 1-\frac{1}{2} \big(\indnorm{1}{\mtx{D}}+\indnorm{\infty}{\mtx{D}} \big) \right]^{-1}.
\end{equation*}
Then, for all $t\ge 0$,
\begin{equation*}
\Prob{\lambda_{\max}\left(\mtx{H}(Z)-\Expect\mtx{H}(Z)\right)\ge t}\le d\cdot \econst^{-t^2/(b\sigma^2)}.
\end{equation*}
Furthermore,
	$$
	\Expect \lambda_{\max}\left( \mtx{H}(Z) - \Expect \mtx{H}(Z) \right) \leq \sigma \sqrt{b\log d}.
	$$ 
\end{cor}

The proof of Corollary~\ref{cor:dob-bound-diff} appears below in Section~\ref{sec:dobrushin-proof}.  In Section~\ref{sec:ising}, we describe an application of the result to physical spin systems in Section~\ref{sec:ising}.  Observe that the bounds here are a factor of $b$ worse than the independent case outlined in Corollary~\ref{cor:bound-diff}.  

\subsection{Proof of Concentration under Dobrushin Assumptions}
\label{sec:dobrushin-proof}

The proof of Corollary~\ref{cor:dob-bound-diff} is longer than the argument behind Corollary~\ref{cor:bound-diff},
but it follows the same pattern.

\paragraph{Exchangeable Counterparts.}

Let $\mtx{X} = \mtx{H}(\Zvec) - \Expect\mtx{H}(\Zvec)$.  To begin,
we form exchangeable counterparts for the random input $\Zvec$
and the random matrix $\mtx{X}$.
\begin{equation*} 
\Zvec' \defby (Z_1, \dots, Z_{J-1}, \tilde{Z}_{J}, Z_{J+1}, \dots, Z_n)
\qtext{and}
\mtx{X}' \defby  \mtx{H}(\Zvec') - \Expect\mtx{H}(\Zvec)
\end{equation*}
where $J$ is an independent index drawn uniformly from $\{1,\dots,n\}$ and
$\tilde{Z}_i$ and $Z_i$ are conditionally i.i.d.\ given $Z_{-i}$ for each
index $i$.

\paragraph{A Kernel Coupling.}

Next, we construct a kernel coupling $(\Zvec_{(i)},\Zvec_{(i)}')_{i\geq 0}$
by adapting the proof of~\cite[Thm.~4.3]{Cha08:Concentration-Inequalities}.
For each $i \geq 1$,
we generate $(\Zvec_{(i)},\Zvec'_{(i)})$ from $(\Zvec_{(i-1)},\Zvec'_{(i-1)})$ by selecting an independent random index $J_i$ uniformly from $\{1,\dots,n\}$ and replacing the $J_i$-th coordinates of $\Zvec_{(i-1)}$ and $\Zvec'_{(i-1)}$ 
with $\tilde{Z}_{(i-1),J_i}$ and $\tilde{Z}'_{(i-1),J_i}$ respectively.
The replacement variables are sampled so that 
$$
Z_{(i-1),j} \indep \tilde{Z}_{(i-1),j} \condl Z_{(i-1),-j}
\quad\text{and}\quad 
Z'_{(i-1),j}\indep \tilde{Z}'_{(i-1),j} \condl Z'_{(i-1),-j}.
$$
We require that $\tilde{Z}_{(i-1),j}$ and $\tilde{Z}'_{(i-1),j}$ are maximally coupled, i.e.,
$$
\Prob{\tilde{Z}_{(i-1),j} \neq \tilde{Z}'_{(i-1),j}\condl \Zvec_{(i-1)}, \Zvec'_{(i-1)}} = \tv\big(\mu_j(\cdot \condl  \Zvec_{(i-1),-j}),\mu_j(\cdot \condl  \Zvec'_{(i-1),-j}) \big).
$$
By construction, the two marginal chains $(\Zvec_{(i)})_{i\geq 0}$ and $(\Zvec'_{(i)})_{i\geq 0}$ have the same the kernel as $(\Zvec,\Zvec')$, and they satisfy the kernel coupling property \eqref{eqn:kernel-coupling}.
Furthermore, the coupling boundedness criterion \eqref{eqn:coupling-premise} is met,
just as in the scalar setting~\cite[p. 78]{Cha08:Concentration-Inequalities}.
Lemma~\ref{lem:kernel-coupling} now implies that $(\mtx{X}, \mtx{X}')$ is a kernel Stein pair with kernel
$$
\mtx{K}(\vct{z},\vct{z}') \defby \sum\nolimits_{i=0}^\infty \Expect\big[\mtx{H}(\Zvec_{(i)}) - \mtx{H}(\Zvec'_{(i)})\condl \Zvec_{(0)}=\vct{z},\Zvec'_{(0)}=\vct{z}' \big].
$$

\paragraph{The Conditional Variances.}

With the kernel coupling established, we may proceed to analyze the conditional variances
$\mtx{V}_{\mtx{X}}$ and $\mtx{V}^{\mtx{K}}$.  First, we collect the information necessary to apply Lemma~\ref{lem:conditional-variance-bound}.
Fix an index $i \geq 0$, and write $\mtx{H}(\Zvec_{(i)}) - \mtx{H}(\Zvec'_{(i)})$ as a telescoping sum:
\begin{multline*}
\mtx{H}(\Zvec_{(i)}) - \mtx{H}(\Zvec'_{(i)})
= \sum_{j=1}^{n} \bigg[\mtx{H}\left(Z_{(i),1},\ldots,Z_{(i),j},Z'_{(i),j+1},\ldots,Z'_{(i),n}\right) \\
- \mtx{H}\left(Z_{(i),1},\ldots,Z_{(i),j-1},Z'_{(i),j},\ldots,Z'_{(i),n}\right) \bigg]
	=: \sum_{j=1}^{n} \mtx{W}_{(i),j}.
\end{multline*}
Introduce the event $\event_{(i),j}\defby \{Z_{(i),j}\ne Z'_{(i),j}\}$. 
Abbreviate $p_{(i),j} = \Prob{\event_{(i),j}\condl \Zvec,\Zvec'}$ and $\tilde{\mtx{W}}_{(i),j} = \Expect[\mtx{W}_{(i),j}\condl \Zvec,\Zvec',\event_{(i),j}] $.
Off of the event $\event_{(i),j}$, it holds that $\mtx{W}_{(i),j}=\zeromtx$. Therefore,
$$
\Expect[\mtx{W}_{(i),j} \condl \Zvec,\Zvec'] = \tilde{\mtx{W}}_{(i),j}\  p_{(i),j}.
$$
In \cite[pp. 77--78]{Cha08:Concentration-Inequalities}, Chatterjee established that, for each $i$ and $j$, 
\begin{align} \label{eqn:l-bound}
p_{(i),j} \le \onevct_j^*\mtx{B}^i \onevct_J 
\qtext{for} \mtx{B}:=\left(1-\frac{1}{n}\right)\Id+\frac{1}{n}\mtx{D}.
\end{align}
We use $\onevct_k$ to denote the $k$th standard basis vector, and $\mtx{B}^i$ refers to the $i$th power of the square, nonnegative matrix $\mtx{B}$.

To continue, make the calculation
\begin{align*}
\left(\sum\nolimits_{j=1}^{n} \Expect[\mtx{W}_{(i),j} \condl Z,Z'] \right)^2
&=\sum\nolimits_{j=1}^{n}\sum\nolimits_{k=1}^{n} \tilde{\mtx{W}}_{(i),j}\tilde{\mtx{W}}_{(i),k}\  p_{(i),j}p_{(i),k}\\
&\psdle\sum\nolimits_{1\le j,k\le n}{\frac{1}{2}(\tilde{\mtx{W}}_{(i),j}^2+\tilde{\mtx{W}}_{(i),k}^2)}{}\ p_{(i),j}p_{(i),k} \\
&\psdle \sum\nolimits_{1\le j,k\le n}\mtx{A}_k^2\cdot \onevct_j^* \mtx{B}^i \onevct_J \cdot\onevct_k^* \mtx{B}^i \onevct_J \\
&= \norm{\mtx{B}^i\onevct_J}_1 \cdot  \sum\nolimits_{k=1}^n\mtx{A}_k^2\cdot \onevct_k^* \mtx{B}^i \onevct_J  \\
&\psdle \indnorm{1}{\mtx{B}}^i \cdot \sum\nolimits_{k=1}^n\mtx{A}_k^2\cdot \onevct_k^* \mtx{B}^i \onevct_J .
\end{align*}
The first semidefinite inequality follows from~\eqref{eqn:matrix-am-gm}.  The second relation depends on~\eqref{eqn:l-bound}.  We reach the next identity by summing over $j$, noting that $\onevct_j^* \mtx{B}^i \onevct_J$ is nonnegative.  The last inequality follows from the definition~\eqref{eqn:induced-norm} of $\indnorm{1}{\cdot}$ and the
fact that this norm is submultiplicative.
Next, take the expectation of the latter display with respect to $J$.  We obtain
\begin{align*}
\Expect\left[\left.\left(\sum\nolimits_{j=1}^{n} \Expect[\mtx{W}_{(i),j} \condl \Zvec,\Zvec'] \right)^2\right|  Z\right]
	&\psdle \indnorm{1}{\mtx{B}}^i \cdot \frac{1}{n}\sum\nolimits_{1\leq j,k\leq n}\mtx{A}_k^2\cdot \onevct_k^* \mtx{B}^i \onevct_j  \\
	&= \indnorm{1}{\mtx{B}}^i \cdot \frac{1}{n}\sum\nolimits_{k=1}^n\mtx{A}_k^2 \cdot\norm{\onevct_k^* \mtx{B}^i}_1  \\
	&\psdle \indnorm{1}{\mtx{B}}^i \cdot \indnorm{\infty}{\mtx{B}}^i \cdot \frac{1}{n}\sum\nolimits_{k=1}^n \mtx{A}_k^2.
\end{align*}
The justifications are similar with those for the preceding calculation.

As a consequence of this bound, we are in a position to apply Lemma~\ref{lem:conditional-variance-bound}.  Set $\mtx{\Gamma}(\Zvec) = n^{-1} \sum_{k=1}^n \mtx{A}_k^2$ and
$s_i = \indnorm{1}{\mtx{B}}^{i/2} \indnorm{\infty}{\mtx{B}}^{i/2}$ for each $i \geq 0$.
The lemma delivers
\begin{align*}
\mtx{V}_{\mtx{X}} &\psdle \frac{1}{2}\mtx{\Gamma}(Z)\psdle \frac{\sigma^2}{2n} \cdot \Id,\quad\text{and} \\
\mtx{V}^{\mtx{K}} &\psdle \frac{1}{2}\left(\sum\nolimits_{i=0}^{\infty} s_i\right)^2 \mtx{\Gamma}(Z) \psdle \left(1-\sqrt{ \indnorm{\infty}{\mtx{B}}\indnorm{1}{\mtx{B}}}\right)^{-2} \frac{\sigma^2}{2n} \cdot \Id.
\end{align*}
where $\sigma^2$ is defined in the statement of Corollary~\ref{cor:dob-bound-diff}.
It remains to simplify the formula for the kernel conditional variance.

The definition of $\mtx{B}$ ensures that
$$
\indnorm{1}{\mtx{B}}=1-\frac{1}{n} \left(1- \indnorm{1}{\mtx{D}}\right)
\quad\text{and}\quad
\indnorm{\infty}{\mtx{B}}=1-\frac{1}{n}\left(1-\indnorm{\infty}{\mtx{D}}\right).
$$
As a consequence of the geometric--arithmetic mean inequality,
$$
1-\sqrt{\indnorm{1}{\mtx{B}} \indnorm{\infty}{\mtx{B}}}
\geq \frac{1}{n} \left[1-\frac{1}{2} \big(\indnorm{1}{\mtx{D}}+\indnorm{\infty}{\mtx{D}} \big) \right].
$$
We conclude that
$$
\mtx{V}^{\mtx{K}}\psdle \left[1-\frac{1}{2}\big(\indnorm{1}{\mtx{D}}+\indnorm{\infty}{\mtx{D}}\big)\right]^{-2} \cdot \frac{n \sigma^2}{2} \cdot \Id
	= \frac{n b^2 \sigma^2}{2} \cdot \Id,
$$
where $b$ is defined in the statement of Corollary~\ref{cor:dob-bound-diff}.

Finally, we invoke \thmref{concentration-bdd} with $c=0$ and $v=b\sigma^2/2$ and $s = nb$ to obtain the advertised conclusions.

\section{Application: Correlation in the 2D Ising Model} \label{sec:ising}

In this section, we apply the dependent matrix bounded differences inequality of Corollary~\ref{cor:dob-bound-diff} to study correlations in a simple spin system.
Consider the 2D Ising model without an external field on an $n \times n$ square lattice with a periodic boundary.  Let $\vct{\sigma} := (\sigma_{ij} : 1 \leq i, j \leq n)$ be an array of random spins taking values in $\{+1, -1\}$.  To simplify the discussion, we treat array indices periodically, so we interpret the index $i$ to mean $((i-1) \bmod n) + 1$.  We also write $(i, j) \sim  (k, l)$ to indicate that the vertices are neighbors in the periodic square lattice; that is, $k = i \pm 1$ and $l = j$ or else $k = i$ and $l = j \pm 1$.  With this notation, the Hamiltonian may be expressed as
$$
H(\vct{\sigma}) = \sum_{(i,j)\sim(k,l)} \sigma_{ij} \, \sigma_{kl},
$$
where the sum occurs over distinct pairs of neighboring vertices.  We assign a probability distribution to the array $\vct{\sigma}$ of spins:
\begin{equation} \label{probdef}
\Prob{ \vct{\sigma}} = \frac{1}{A} \exp\big( \beta \, H(\vct{\sigma}) \big),
\end{equation}
where $A = \sum_{\vct{\sigma'}} \exp\big( \beta \, H(\vct{\sigma}')\big)$ denotes the normalizing constant (also known as the partition function).  This model has been studied extensively, and it is known to exhibit a phase transition at $\beta_c = \tfrac{1}{2} \log(1 + \sqrt{2})$.  For example, see~\cite{Onsager}.

Fix a positive number $d \leq n$.  For indices $1 \leq i, j \leq d$, we define the spin--spin correlation function as
$$
c_{ij} = \Expect[ \sigma_{11} \, \sigma_{ij} ].
$$
We write $\mtx{C}$ for the $d\times d$ matrix whose entries are $c_{ij}$.  The paper~\cite{wu1976spin} of Wu offers an explicit expression for the correlations in the limit as the size $n$ of the lattice tends to infinity.  In particular, in the high-temperature regime $\beta < \beta_c$, the correlations decay exponentially.  On the other hand, this is a limiting result and there is no analytic formula for finite lattices.

One may wish to estimate the spin--spin correlation matrix from a sampled value $\vct{\sigma}$ of the spins.  We propose the estimator
\begin{equation} \label{hatcdef}
\widehat{C}_{ij} := \frac{1}{n^2} \sum_{1 \leq k, l \leq n}
	\sigma_{kl} \cdot \sigma_{k+i-1,l+j-1}
	\quad\text{for $1\leq i,j\leq d$.}
\end{equation}
The mean of $\widehat{\mtx{C}}$ is the spin--spin correlation matrix $\mtx{C}$, so it is natural to wonder about the deviations of the estimator from its mean value.  We can use the concentration results from the previous section to quantify these fluctuations.

\subsection{Concentration for General Matrices}

Since the estimator $\widehat{\mtx{C}}$ need not be Hermitian, we need a way to extend our techniques to general matrices.  We employ a well-known device from operator theory, called the \term{Hermitian dilation}~\cite[Sec.~2.6]{Tro11:User-Friendly-FOCM}.

\begin{defn}[Hermitian dilation]
Consider a matrix $\mtx{B}\in \C^{d_1\times d_2}$, and set $d=d_1+d_2$.
The \emph{Hermitian dilation} of $\mtx{B}$ is the matrix
\begin{equation*}
\D(\mtx{B}):=\begin{bmatrix} \zeromtx & \mtx{B}\\ \mtx{B^*} & \zeromtx\end{bmatrix} \in \Sym{d}.
\end{equation*}
\end{defn}

\noindent
The dilation preserves spectral properties in the sense that
$\lambda_{\max}(\D(\mtx{B})) = \norm{\D(\mtx{B})} = \norm{\mtx{B}}$.
Therefore,
\begin{equation} \label{dilationeq}
\Prob{ \smnorm{}{\widehat{\mtx{C}} - \mtx{C}} \geq t }
	= \Prob{ \lambda_{\max}\big( \D(\widehat{\mtx{C}}) - \D(\mtx{C}) \big) \geq t }.
\end{equation}
Using this observation, we can obtain a tail bound for the spectral-norm error in the estimator $\widehat{\mtx{C}}$ by studying its dilation.

\subsection{Bounding the Dobrushin Coefficients}

To apply Corollary~\ref{cor:dob-bound-diff}, we need to bound the Dobrushin coefficients of the array $\vct{\sigma}$ of spins.  Let $\vct{\sigma}' \in \{ \pm 1 \}^{n\times n}$ be a second independent draw from the Ising model.  Extending our notation from before, we write $\mu_{ij}(\cdot \condl \vct{\sigma}_{-(i,j)})$ for the conditional distribution of $\sigma_{ij}$ given the remaining variables.  In our setting,
\begin{multline} \label{eqn:tv-ising}
\tv\big( \mu_{ij}(\cdot \condl \vct{\sigma}_{-(i,j)}), \
	\mu_{ij}(\cdot \condl \vct{\sigma}'_{-(i,j)}) \big) \\
	= \abs{ \Prob{ \sigma_{ij} = 1 \,\bigg\vert\, \sum\nolimits_{(k,l): (i,j) \sim (k,l)} \sigma_{kl} } -
	\Prob{ \sigma'_{ij} = 1 \,\bigg\vert\, \sum\nolimits_{(k,l): (i,j) \sim (k,l)} \sigma'_{kl} } }.
\end{multline}
It follows from~\eqref{probdef} that
$$
\Prob{ \sigma_{ij} = 1 \,\bigg\vert\, \sum\nolimits_{(k,l) : (i,j) \sim (k,l)} \sigma_{kl} = s}
	= \frac{\exp(s\beta)}{\exp(s\beta) + \exp(-s\beta)} = \frac{1}{1 + \exp(-2s\beta)}
$$
for each possible value $s \in \{-4,-2,0,2,4\}$.  Therefore, the expression~\eqref{eqn:tv-ising} admits the upper bound
$$
\frac{1}{1 + \exp(-4\beta)} - \frac{1}{2}
$$
when $\vct{\sigma}$ and $\vct{\sigma}'$ differ in a single coordinate.  We may select the Dobrushin interdependence matrix
$$
D_{(i,j),(k,l)} = \begin{cases}
	(1 + \exp(-4\beta))^{-1} - \tfrac{1}{2}, & \text{when} (i,j)\sim(k,l) \\
	0, & \text{otherwise}.
\end{cases}
$$
This matrix satisfies the Dobrushin condition~\eqref{eqn:dobrushin}.  By direct computation,
\begin{equation} \label{Dbound}
\max\big\{ \indnorm{1}{\mtx{D}}, \ \indnorm{\infty}{\mtx{D}} \big\}
	\leq \frac{4}{1 - \exp(-4\beta)} - 2.
\end{equation}
because every vertex has four neighbors.  The right-hand side of~\eqref{Dbound} is smaller than one precisely when $\beta < \beta_D = \tfrac{1}{4} \log(3)$.  Since $\beta_D < \beta_c$, the hypotheses of Corollary~\ref{cor:dob-bound-diff} are satisfied for only part of the high-temperature regime.

\subsection{Tail Bound for the Estimator}

We intend to apply Corollary~\ref{cor:dob-bound-diff} to the Hermitian matrix $\D(\widehat{\mtx{C}})$.  For each index $1 \leq i,j \leq n$, write $\widehat{\mtx{C}}^{ij}$ for the value of $\widehat{\mtx{C}}$ when the sign of $\sigma_{ij}$ is flipped.  From~\eqref{hatcdef}, we have the inequalities
$$
\abs{ \widehat{C}_{kl} - \widehat{C}_{kl}^{ij} } \leq\frac{4}{n^2}
\quad\text{for $1 \leq k, l \leq d$.}
$$
As a consequence, we reach the semidefinite relation
$$
\big( \D(\widehat{\mtx{C}}) - \D(\widehat{\mtx{C}}_{ij}) \big)^2
	\psdle \frac{16 d^2}{n^4} \Id.
$$
Summing over all vertices in the lattice, we obtain an inequality for the boundedness parameter
$$
\sigma^2 = \norm{ \sum\nolimits_{1\leq i,j\leq n} \frac{16d^2}{n^4} \cdot \Id }
	= \frac{16d^2}{n^2}.
$$
For $\beta < \beta_D$, Corollary~\ref{cor:dob-bound-diff} implies that
$$
\Prob{ \smnorm{}{\widehat{\mtx{C}} - \mtx{C}} \geq t }
	\leq 2d \cdot \exp\left( \frac{-t^2}{3 - 4 (1 + \exp(-4\beta))^{-1}} \cdot
	\frac{n^2}{16d^2} \right).
$$
Therefore, the typical deviation $\Expect \smnorm{}{\widehat{\mtx{C}} - \mtx{C}}$
has order $(d\sqrt{\log d})/n$.  Therefore, in the regime where $\beta < \beta_D$, one sample suffices to obtain an accurate estimate of the spin--spin correlation matrix $\mtx{C}$, provided that $n \gg d$.

\section{Complements} \label{sec:beyond}

The tools of Section~\ref{sec:concentration-bdd} are applicable in a wide variety of settings.  To indicate what might be possible, we briefly present another packaged concentration result.  We also indicate some prospects for future research.

\subsection{Matrix-Valued Functions of Haar Random Elements}

This section describes a concentration result for a matrix-valued function of a random element drawn uniformly from a compact group.  This corollary can be viewed as a matrix extension of~\cite[Thm.~4.6]{Cha08:Concentration-Inequalities}.  We provide the proof in Appendix~\ref{sec:haar}.

\begin{cor}[Concentration for Hermitian Functions of Haar Measures] \label{cor:concentration-haar}
	Let $Z\sim \mu$ be Haar distributed on a compact topological group $G$, and 
	let $\mtx{\Psi} : G \to \Sym{d}$ be a measurable function satisfying $\Expect{\mtx{\Psi}(Z)} = \mtx{0}$.
	Let $Y, Y_1, Y_2, \dots$ be i.i.d.~random variables in $G$ satisfying 
	\begin{align} \label{eqn:Y-property}
	Y =_d Y^{-1} \qtext{and} zYz^{-1} =_d Y \qtext{for all } z \in G.
	\end{align}
	Compute the boundedness parameter
	$$
	\sigma^2 \defby \frac{S^2}{2}\sum\nolimits_{i=0}^\infty \min\big\{1,\ 4RS^{-1} \tv(\mu_i, \mu) \big\}
	$$
	where $\mu_i$ is the distribution of the product ${Y}_i\cdots {Y}_1$,
	$$
	\norm{\mtx{\Psi}(z)} \leq R \qtext{for all} z\in G, \qtext{and} S^2 = \sup_{g \in G}\norm{\Expect\big[(\mtx{\Psi}(g) - \mtx{\Psi}(Yg))^2 \big]}.
	$$
	Then, for all $t\geq 0$, 
	$$
	\Prob{ \lambda_{\max}\left( \mtx{\Psi}(Z) \right) \geq t }
		\leq d \cdot \econst^{-t^2/(2\sigma^2)}.
	$$
	Furthermore,
	$$
	\Expect \lambda_{\max}\left( \mtx{\Psi}(Z) \right) \leq \sigma \sqrt{2\log d}.
	$$ 
\end{cor}

\cororef{concentration-haar} relates the concentration of Hermitian functions to the convergence of random walks on a group.
Since representation theory leads to a matrix model of compact groups, it is often natural to build random matrices from a group representation.
In particular, \cororef{concentration-haar} can be used to study matrices constructed from random permutations or random unitary matrices.
We omit the details.

\subsection{Conjectures and Consequences}\label{sec:conjectures}

Lugosi et al.~\cite{LugosiSelfBounding} study a class of self-bounding (scalar) functions, which arise in applications in statistics and learning theory.  They use log-Sobolev inequalities to obtain information about the concentration properties of these functions.  It is also possible to perform the analysis using the method of exchangeable pairs.

Let us introduce the matrix analog of a self-bounding function.

\begin{defn}[Self-bounding Matrix Function] \label{def:self-bounding}
A function $\mtx{H} : \metricspace\to\Sym{d}$ is called \term{$(a,b)$ matrix self-bounding} if, for any $\Zvec, \Zvec' \in \metricspace$,
\begin{enumerate}
\item $\mtx{H}(\Zvec)-\mtx{H}(z_1,\ldots,z_i',\ldots,z_n)\psdle \Id$, and
\item $\sum_{i=1}^n(\mtx{H}(\Zvec)-\mtx{H}(z_1,\ldots,z_i',\ldots,z_n))_+\psdle a \mtx{H}(\Zvec)+b \Id$.
\end{enumerate}
$\mtx{H}$ is \term{weakly $(a,b)$ matrix self-bounding} if, for any $\Zvec,\Zvec'\in\metricspace$, 
$$
\sum\nolimits_{i=1}^n(\mtx{H}(\Zvec)-\mtx{H}(z_1,\ldots,z_i',\ldots,z_n))_+^2\psdle a \mtx{H}(\Zvec)+ b \Id.
$$
\end{defn}

\noindent
Mackey~\cite[Thm.~25]{Mackeythesis} proposed a slightly different definition that includes an additional self-reciprocity condition.  His analysis requires this extra hypothesis because it is based on matrix Stein pairs.

The approach in this paper is not quite strong enough to develop concentration inequalities for self-bounding matrix functions.  Our techniques would work if the following mean value trace inequality were valid.

\begin{conjecture}[Signed Mean Value Trace Inequalities] \label{conj:mvti}
For all matrices $\mtx{A}, \mtx{B}, \mtx{C} \in \Sym{d}$, all positive integers $q$,  and any $s > 0$ it holds that
\begin{align*}
\trace \big[\mtx{C} (\econst^{\mtx{A}} - \econst^{\mtx{B}} )\big] \leq  
	\frac{1}{2} \trace\big[(s(\mtx{A}-\mtx{B})_+^2+s^{-1} \, \mtx{C}_+^2)
	\econst^{\mtx{A}}+(s(\mtx{A}-\mtx{B})_-^2+s^{-1} \mtx{C}_-^2)\econst^{\mtx{B}}) \big] .
\end{align*}
and
\begin{align*}
\trace \big[\mtx{C} (\mtx{A}^{q} - \mtx{B}^{q} )\big] \leq  
	\frac{q}{2} \big[(s(\mtx{A}-\mtx{B})_+^2+s^{-1}\mtx{C}_+^2)\abs{\mtx{A}}^{q-1} + (s(\mtx{A}-\mtx{B})_-^2+s^{-1}\mtx{C}_-^2)\abs{\mtx{B}}^{q-1}) \big].
\end{align*}
\end{conjecture}

\noindent
This statement involves the standard matrix functions that lift the scalar functions $x_+ := \max\{ x, 0 \}$ and $x_- := \max\{-a, 0\}$.  Extensive simulations with random matrices suggest that Conjecture~\ref{conj:mvti} holds, but we did not find a proof.

\section*{Acknowledgements}\label{sec:acknowledgements}
Paulin thanks his thesis advisors, Louis Chen and Adrian R\"ollin, for their helpful comments on this manuscript.
Tropp was supported by ONR awards N00014-08-1-0883 and N00014-11-1002, AFOSR award FA9550-09-1-0643, and a Sloan Research Fellowship.

\appendix
\section{Operator Inequalities} \label{sec:operator-ineq}

Our main results rely on some basic inequalities from operator theory.
We are not aware of good references for this material, so we have
included short proofs.

\subsection{Young's Inequality for Commuting Operators}

In the scalar setting, Young's inequality provides an additive
bound for the product of two numbers.  More precisely,
for indices $p, q \in (1, \infty)$ that satisfy the
conjugacy relation $p^{-1} + q^{-1} = 1$, we have
\begin{equation} \label{eqn:young-scalar}
ab \leq \frac{1}{p} \abs{a}^p + \frac{1}{q} \abs{b}^q
\quad\text{for all $a, b \in \R$.}
\end{equation}
The same result has a natural extension for commuting operators.

\begin{lemma}[Young's Inequality for Commuting Operators]
\label{lem:young-commute}
Suppose that $\mathcal{A}$ and $\mathcal{B}$ are self-adjoint linear maps on
the Hilbert space $\M^d$ that commute with each other.
Let $p, q \in (1, \infty)$ satisfy the conjugacy relation
$p^{-1} + q^{-1} = 1$.  Then
$$
\mathcal{A}\mathcal{B}
	\psdle \frac{1}{p} \abs{\mathcal{A}}^p + \frac{1}{q} \abs{\mathcal{B}}^q.
$$
\end{lemma}

\begin{proof}
Since $\mathcal{A}$ and $\mathcal{B}$ commute,
there exists a unitary operator $\mathcal{U}$
and diagonal operators $\mathcal{D}$ and $\mathcal{M}$ for which
$\mathcal{A} = \mathcal{U}\mathcal{D}\mathcal{U}^*$ and
$\mathcal{B} = \mathcal{U}\mathcal{M}\mathcal{U}^*$.
Young's inequality~\eqref{eqn:young-scalar} for scalars
immediately implies that
$$
\mathcal{D}\mathcal{M} \psdle
	\frac{1}{p}\abs{\mathcal{D}}^p + \frac{1}{q}\abs{\mathcal{M}}^q.
$$
Conjugating both sides of this inequality by $\mathcal{U}$, we obtain
$$
\mathcal{A}\mathcal{B}
	= \mathcal{U}(\mathcal{D}\mathcal{M})\mathcal{U}^* 
	\psdle \frac{1}{p}\mathcal{U} \abs{\mathcal{D}}^p \mathcal{U}^*
	+ \frac{1}{q}\mathcal{U}\abs{\mathcal{M}}^q\mathcal{U}^*
	= \frac{1}{p}\abs{\mathcal{A}}^p + \frac{1}{q}\abs{\mathcal{B}}^q.
$$
The last identity follows from the definition of a standard function of an
operator.
\end{proof}

\subsection{An Operator Version of Cauchy--Schwarz}

We also need a simple version of the Cauchy--Schwarz inequality
for operators.  The proof follows a classical argument, but it
also involves an operator decomposition.

\begin{lemma}[Operator Cauchy--Schwarz] \label{lem:operator-cs}
Let $\mathcal{A}$ be a self-adjoint linear operator on the Hilbert space $\M^d$,
and let $\mtx{M}$ and $\mtx{N}$ be matrices in $\M^d$.  Then
$$
\absip{ \mtx{M} }{ \mathcal{A}(\mtx{N}) }
	\leq \big[ \ip{ \mtx{M} }{ \abs{\mathcal{A}}(\mtx{M}) } \cdot
	\ip{ \mtx{N} }{ \abs{\mathcal{A}}(\mtx{N}) } \big]^{1/2}.
$$
The inner product symbol refers to the trace, or Frobenius, inner product.
\end{lemma}

\begin{proof}
Consider the Jordan decomposition $\mathcal{A} = \mathcal{A}_{+} - \mathcal{A}_{-}$,
where $\mathcal{A}_+$ and $\mathcal{A}_{-}$ are both positive semidefinite.  For all $s > 0$,
\begin{align*}
0 &\leq \ip{(s \mtx{M} - s^{-1} \mtx{N}) }{ \mathcal{A}_+(s \mtx{M} - s^{-1}\mtx{N}) } \\
	&= s^2 \ip{ \mtx{M} }{\mathcal{A}_+(\mtx{M}) }
	+ s^{-2} \ip{ \mtx{N} }{ \mathcal{A}_+(\mtx{N}) }
	- 2 \ip{ \mtx{M} }{ \mathcal{A}_+( \mtx{N} ) }.
\end{align*}
Likewise,
\begin{align*}
0 &\leq \ip{(s \mtx{M} + s^{-1} \mtx{N}) }{ \mathcal{A}_{-}(s \mtx{M} + s^{-1}\mtx{N}) } \\
	&= s^2 \ip{ \mtx{M} }{\mathcal{A}_{-}(\mtx{M}) }
	+ s^{-2} \ip{ \mtx{N} }{ \mathcal{A}_{-}(\mtx{N}) }
	+ 2 \ip{ \mtx{M} }{ \mathcal{A}_{-}( \mtx{N} ) }. 
\end{align*}
Add the latter two inequalities and rearrange the terms to obtain
$$
2 \ip{ \mtx{M} }{ \mathcal{A}(\mtx{N}) }
	\leq s^2 \ip{ \mtx{M} }{ \abs{\mathcal{A}}(\mtx{M}) }
	+ s^{-2} \ip{ \mtx{N} }{ \abs{\mathcal{A}}(\mtx{N}) },
$$
where we have used the relation $\abs{\mathcal{A}} = \mathcal{A}_+ + \mathcal{A}_-$.  Take the infimum of the right-hand side over $s > 0$ to reach
\begin{equation} \label{eqn:cs-halfway}
\ip{ \mtx{M} }{ \mathcal{A}( \mtx{N} ) }
	\leq \big[ \ip{ \mtx{M} }{\abs{\mathcal{A}}(\mtx{M}) } \cdot 
	\ip{ \mtx{M} }{ \abs{\mathcal{A}}( \mtx{N} ) } \big]^{1/2}.
\end{equation}
Repeat the same argument, interchanging the roles of the matrices
$s \mtx{M} - s^{-1} \mtx{N}$ and $s \mtx{M} + s^{-1} \mtx{N}$.
We conclude that~\eqref{eqn:cs-halfway} also holds with an absolute value on the
left-hand side.  This observation completes the proof.
\end{proof}

\section{Proof of the Exponential Tail Bound} \label{sec:proof-concentration-bdd}

This appendix contains a proof of the exponential tail bound \thmref{concentration-bdd}.  The argument parallels the approach developed in~\cite{MackeyJoChFaTr12}, but we require more powerful estimates along the way.  In view of the similarities, we emphasize the places where the proofs differ, and we suppress details that are identical with the earlier work.

\subsection{The Matrix Laplace Transform Method} \label{sec:matrix-laplace}

A central tool in our investigation is a matrix variant of the classical
moment generating function.  Ahlswede \& Winter~\cite[App.]{AW02:Strong-Converse} introduced this definition in their investigation of matrix concentration.

\begin{defn}[Trace Mgf]
Let $\mtx{X}$ be a random Hermitian matrix.  The
\term{(normalized) trace moment generating function} of $\mtx{X}$ is defined as
$$
m(\theta) \defby m_{\mtx{X}}(\theta) \defby \Expect \ntr \econst^{\theta \mtx{X}}
\quad\text{for $\theta \in \R$.}
$$
\end{defn}

The following proposition from~\cite[Prop.~3.3]{MackeyJoChFaTr12} collects results from~\cite{AW02:Strong-Converse,Oli10:Sums-Random,Tro11:User-Friendly-FOCM,CGT11:Masked-Sample}.

\begin{prop} [Matrix Laplace Transform Method] \label{prop:matrix-laplace}
Let $\mtx{X} \in \Sym{d}$ be a random matrix with normalized trace mgf
$m(\theta) \defby \Expect \ntr \econst^{\theta \mtx{X}}$.  For each $t \in \R$, 
\begin{align} 
\Prob{\lambda_{\max} ( \mtx{X}) \geq t}
	&\leq d \cdot \inf_{\theta > 0} \ \exp\{ -\theta t + \log m(\theta) \}.
	\label{eqn:laplace-upper-tail} \\
\Prob{\lambda_{\min}(\mtx{X}) \leq t}
	&\leq d \cdot \inf_{\theta < 0} \ \exp\{ -\theta t + \log m(\theta) \}.
	\label{eqn:laplace-lower-tail}
\end{align}
Furthermore,
\begin{align}
\Expect \lambda_{\max}(\mtx{X})
	\leq \inf_{\theta > 0} \ \frac{1}{\theta}\, [\log d + \log m(\theta)].  \label{eqn:laplace-upper-mean} \\
\Expect \lambda_{\min}(\mtx{X})
	\geq \sup_{\theta < 0} \ \frac{1}{\theta}\, [\log d + \log m(\theta)].  \label{eqn:laplace-lower-mean}
\end{align}
\end{prop}

\noindent
In summary, we can bound the extreme eigenvalues of a random matrix
by controlling the trace mgf.

\subsection{The Method of Exchangeable Pairs} \label{sec:method-exchange}

The main technical challenge in developing concentration inequalities is to
obtain bounds for the trace mgf.  In this work, we follow the approach from
the paper~\cite{MackeyJoChFaTr12}, which extends Chatterjee's concentration
argument~\cite{Cha07:Steins-Method} to the matrix setting.  The key idea is
to use an exchangeable pair to bound the derivative of the trace mgf, which
in turns allows us to control the growth of the trace mgf.

We begin with a technical lemma, which generalizes \cite[Lem.~2.3]{MackeyJoChFaTr12} and \cite[Lem.~3.1]{Cha08:Concentration-Inequalities}.
This result permits us to rewrite certain matrix expectations
using kernel Stein pairs.

\begin{lemma}[Method of Exchangeable Pairs] \label{lem:exchange}
Suppose that $(\mtx{X}, \mtx{X}') \in \Sym{d} \times \Sym{d}$ is a $\mtx{K}$-Stein pair constructed from an auxiliary exchangeable pair $(Z,Z')$.
Let $\mtx{F}: \Sym{d} \rightarrow \Sym{d}$ be a measurable function that satisfies the regularity condition
\begin{equation} \label{eqn:regularity-mep}
\Expect \norm{ \mtx{K}(Z,Z') \cdot \mtx{F}(\mtx{X})  } < \infty.
\end{equation}
Then
\begin{equation}  \label{eqn:exchange}
 \Expect \left[ \mtx{X} \cdot \mtx{F}(\mtx{X}) \right]
 	= \frac{1}{2} \Expect \left[ \mtx{K}(Z,Z')(\mtx{F}(\mtx{X}) - \mtx{F}(\mtx{X}') ) \right].
\end{equation}
\end{lemma}
\begin{proof}
\defref{K-stein-pair}, of a kernel Stein pair, implies that
$$
\Expect[ \mtx{X} \cdot \mtx{F}(\mtx{X}) ]
	= \Expect \big[ \Expect[ \mtx{K}(Z,Z') \condl Z ] \cdot \mtx{F}(\mtx{X}) \big]
	= \Expect[ \mtx{K}(Z,Z') \, \mtx{F}(\mtx{X}) ],
$$
where we justify the pull-through property of conditional expectation
using the regularity condition~\eqref{eqn:regularity-mep}.
Since the kernel $\mtx{K}$ satisfies the antisymmetry property~\eqref{eqn:K-stein-pair},
we also have the relation
$$
\Expect[ \mtx{K}(Z,Z')\, \mtx{F}(\mtx{X}) ]
	= \Expect[\mtx{K}(Z',Z) \, \mtx{F}(\mtx{X}') ]
	= - \Expect[ \mtx{K}(Z,Z') \, \mtx{F}(\mtx{X}') ].
$$
Average the two preceding displays to reach the identity~\eqref{eqn:exchange}.
\end{proof}

Under suitable regularity conditions, the derivative of the trace mgf of a random matrix $\mtx{X}$ has precisely the form needed to invoke to the method of exchangeable pairs:
$$
m'(\theta) = \Expect \ntr \big[ \mtx{X} \econst^{\theta \mtx{X}} \big].
$$
Hence, we may apply \lemref{exchange} with $\mtx{F}(\mtx{X}) = \econst^{\theta \mtx{X}}$ to obtain the expression
\begin{equation} \label{eqn:m-prime-temp}
m'(\theta) = \frac{1}{2}
\Expect \ntr \big[ \mtx{K}(Z,Z') \big( \econst^{\theta \mtx{X}} - \econst^{\theta \mtx{X}'} \big) \big].
\end{equation}
The primary novelty in this work is a method for bounding the right-hand side of~\eqref{eqn:m-prime-temp}.

\subsection{The Exponential Mean Value Trace Inequality} \label{sec:emvti}

To control the expression~\eqref{eqn:m-prime-temp}
for the derivative of the trace mgf, we will invoke
Lemma~\ref{lem:emvti}, the exponential mean value trace inequality.
We establish this key lemma in this section.
See the manuscript \cite{Anote} for an alternative
proof.

\begin{proof}[Proof of Lemma~\ref{lem:emvti}] 
To begin, we develop an alternative expression for the trace
quantity that we need to bound.  Observe that 
$$
\frac{\diff{}}{\diff{\tau}} \econst^{\tau \mtx{A}} \econst^{(1-\tau)\mtx{B}}
	= \econst^{\tau \mtx{A}}(\mtx{A} - \mtx{B}) \econst^{(1-\tau)\mtx{B}}.
$$
The Fundamental Theorem of Calculus delivers the identity
$$
\econst^{\mtx{A}} - \econst^{\mtx{B}}
	= \int_0^1 \frac{\diff{}}{\diff{\tau}}
	\econst^{\tau \mtx{A}} \econst^{(1-\tau)\mtx{B}} \idiff{\tau}
	= \int_0^1 \econst^{\tau \mtx{A}}(\mtx{A} - \mtx{B}) \econst^{(1-\tau)\mtx{B}}
	\idiff{\tau}.
$$
Therefore, using the definition of the trace inner product, we reach
\begin{equation} \label{eqn:emvti-integral}
\trace \big[ \mtx{C} \big( \econst^{\mtx{A}} - \econst^{\mtx{B}} \big) \big]
	= \int_0^1 \ip{ \mtx{C} }{ \econst^{\tau \mtx{A}}(\mtx{A} - \mtx{B}) \econst^{(1-\tau)\mtx{B}} } \diff{\tau}.
\end{equation}
We can bound the right-hand side by developing an appropriate
matrix version of the inequality between the logarithmic
mean and the arithmetic mean.

Let us define two families of positive-definite
operators on the Hilbert space $\M^{d}$:
$$
\mathcal{A}_{\tau}(\mtx{M}) = \econst^{\tau \mtx{A}} \mtx{M}
\quad\text{and}\quad
\mathcal{B}_{1-\tau}(\mtx{M}) = \mtx{M} \econst^{(1-\tau) \mtx{B}}
\quad\text{for each $\tau \in [0,1]$.}
$$
In other words, $\mathcal{A}_{\tau}$ is a left-multiplication operator,
and $\mathcal{B}_{1-\tau}$ is a right-multiplication operator.  It follows
immediately that $\mathcal{A}_{\tau}$ and $\mathcal{B}_{1-\tau}$ commute
for each $\tau \in [0,1]$.
Young's inequality for commuting operators, Lemma~\ref{lem:young-commute},
implies that
\begin{equation*} 
\mathcal{A}_\tau\mathcal{B}_{1-\tau} 
	\psdle \tau \cdot \abs{\mathcal{A}_\tau}^{1/\tau} + (1-\tau) \cdot \abs{\mathcal{B}_{1-\tau}}^{1/(1-\tau)}
	=  \tau \cdot \abs{\mathcal{A}_1} + (1-\tau) \cdot \abs{\mathcal{B}_{1}}.
\end{equation*}
Integrating over $\tau$, we discover that
\begin{align}\label{eqn:young-exp-bound}
\int\nolimits_{0}^1\mathcal{A}_\tau\mathcal{B}_{1-\tau} d\tau
	&\psdle \frac{1}{2} (|\mathcal{A}_1| + |\mathcal{B}_{1}|)
	= \frac{1}{2} (\mathcal{A}_1 + \mathcal{B}_{1}).
\end{align}
This is our matrix extension of the logarithmic--arithmetic mean inequality.

To relate this result to the problem at hand, we rewrite
the expression~\eqref{eqn:emvti-integral} using the operators
$\mathcal{A}_{\tau}$ and $\mathcal{B}_{1-\tau}$.  Indeed,
\begin{multline}
\trace \big[ \mtx{C} \big( \econst^{\mtx{A}} - \econst^{\mtx{B}} \big) \big]
	= \int_0^1 \ip{ \mtx{C} }{ (\mathcal{A}_{\tau} \mathcal{B}_{1-\tau} )(\mtx{A} - \mtx{B}) } \diff{\tau} \\
	\leq \left[ \int_0^1 \ip{ \mtx{C} }{ (\mathcal{A}_{\tau} \mathcal{B}_{1-\tau} )(\mtx{C})} \diff{\tau} \cdot
	\int_0^1 \ip{ \mtx{A} - \mtx{B} }{ (\mathcal{A}_{\tau} \mathcal{B}_{1-\tau} )(\mtx{A} - \mtx{B})} \diff{\tau} \right]^{1/2}.
	\label{eqn:exp-bd-integrals}
\end{multline}
The second identity follows from the definition of the trace inner product.
The last relation follows from the operator Cauchy--Schwarz inequality, Lemma~\ref{lem:operator-cs}, and the usual Cauchy--Schwarz inequality for the integral.

It remains to bound the two integrals in~\eqref{eqn:exp-bd-integrals}.  These estimates are an immediate consequence of~\eqref{eqn:young-exp-bound}.  First,
\begin{multline} \label{eqn:exp-integral-1}
\int_0^1 \ip{ \mtx{C} }{ (\mathcal{A}_{\tau} \mathcal{B}_{1-\tau} )(\mtx{C})} \idiff{\tau}
	\leq \frac{1}{2} \ip{ \mtx{C} }{ (\mathcal{A}_1 + \mathcal{B}_1)(\mtx{C}) } \\
	= \frac{1}{2} \ip{ \mtx{C} }{ \econst^{\mtx{A}} \mtx{C} + \mtx{C} \econst^{\mtx{B}} }
	= \frac{1}{2} \trace\big[ \mtx{C}^2 \big( \econst^{\mtx{A}} + \econst^{\mtx{B}} \big) \big]. 
\end{multline}
The last two relations follow from the definitions of the operators $\mathcal{A}_1$ and $\mathcal{B}_1$, the definition of the trace inner product, and the cyclicity of the trace.  Likewise,
\begin{equation} \label{eqn:exp-integral-2}
\int_0^1 \ip{ \mtx{A} - \mtx{B} }{ (\mathcal{A}_{\tau} \mathcal{B}_{1-\tau} )(\mtx{A} - \mtx{B})} \diff{\tau}
= \frac{1}{2} \trace\big[ (\mtx{A}-\mtx{B})^2 \big( \econst^{\mtx{A}} + \econst^{\mtx{B}} \big) \big].
\end{equation}
Substitute~\eqref{eqn:exp-integral-1} and~\eqref{eqn:exp-integral-2} into the inequality~\eqref{eqn:exp-bd-integrals} to reach
$$
\trace \big[ \mtx{C} \big( \econst^{\mtx{A}} - \econst^{\mtx{B}} \big) \big]
	\leq \frac{1}{2} \bigg( \trace\big[ \mtx{C}^2 \big( \econst^{\mtx{A}} + \econst^{\mtx{B}} \big) \big] \cdot
	\trace\big[ (\mtx{A}-\mtx{B})^2 \big( \econst^{\mtx{A}} + \econst^{\mtx{B}} \big) \big] \bigg)^{1/2}.
$$
We obtain the result stated in Lemma~\ref{lem:emvti} by applying
the numerical inequality between the geometric mean and the
arithmetic mean.
\end{proof}

\subsection{Bounding the Derivative of the Trace Mgf} \label{sec:d-trace-mgf}

We are now prepared to obtain a bound for the derivative of the trace mgf
in terms of the conditional variance and the kernel conditional variance.

\begin{lemma}[The Derivative of the Trace Mgf] \label{lem:mgf-derivative}
Suppose that $(\mtx{X}, \mtx{X}') $ is a $\mtx{K}$-Stein pair,
and assume that $\mtx{X}$ is almost surely bounded in norm. 
Define the normalized trace mgf
$m(\theta) \defby \Expect \ntr \econst^{\theta \mtx{X}}$.
Then
\begin{align}
\abs{m'(\theta)}
	&\leq  \frac{1}{2} \abs{\theta} \cdot \inf_{s > 0}\ \Expect \ntr\big[\big(s\mtx{V}_{\mtx{X}} + s^{-1} \mtx{V}^{\mtx{K}} \big) \, \econst^{\theta \mtx{X}} \big]
	\quad\text{for all $\theta \in \R$.} \label{eqn:m-prime-Delta} 
\end{align}
The conditional variances $\mtx{V}_{\mtx{X}}$ and $\mtx{V}^{\mtx{K}}$ are defined in~\eqref{eqn:conditional-variance} and \eqref{eqn:K-conditional-variance}.
\end{lemma}

\begin{proof}
Consider the derivative of the trace mgf
\begin{equation} \label{eqn:m-prime-v1}
m'(\theta)
	= \Expect \ntr \left[ \ddt{\theta} \, \econst^{\theta \mtx{X}} \right]
	= \Expect \ntr \big[ \mtx{X} \econst^{\theta \mtx{X}} \big],
\end{equation}
where the dominated convergence theorem and the boundedness of $\mtx{X}$
justify the exchange of expectation and derivative.
When $\theta = 0$, we have $m'(\theta) = 0$, as advertised.
When $\theta \neq 0$, the form of this derivative is ripe for an application of the method of exchangeable pairs, Lemma~\ref{lem:exchange}.
Since $\mtx{X}$ is bounded, the regularity condition~\eqref{eqn:regularity-mep} is satisfied, and we obtain
\begin{equation} \label{eqn:m-prime-v2}
m'(\theta) = \frac{1}{2}
	\Expect \ntr \big[ \mtx{K}(Z,Z') \big(\econst^{\theta \mtx{X}} - \econst^{\theta \mtx{X}'} \big) \big].
\end{equation}
The exponential mean value trace inequality, Lemma~\ref{lem:emvti}, implies that
\begin{align*}
|m'(\theta)|
&\leq \frac{1}{8} \cdot \inf_{s > 0}\ \Expect \ntr \big[
	\big(s \, (\theta\mtx{X}-\theta\mtx{X}')^2+s^{-1}\mtx{K}(Z,Z')^2 \big) \cdot \big(\econst^{\theta \mtx{X}} + \econst^{\theta \mtx{X}'} \big) \big] \\
&= \frac{1}{4} \cdot \inf_{s > 0}\ \Expect \ntr \big[
	\big(s \, (\theta\mtx{X}-\theta\mtx{X}')^2+s^{-1}\mtx{K}(Z,Z')^2 \big) \cdot \econst^{\theta \mtx{X}} \big] \\
&= \frac{1}{4} \abs{\theta} \cdot \inf_{t > 0}\ \Expect \ntr \big[
	\big(t \, (\mtx{X}-\mtx{X}')^2+t^{-1} \mtx{K}(Z,Z')^2 \big) \cdot \econst^{\theta \mtx{X}} \big] \\
&=  \frac{1}{2} \abs{\theta} \cdot \inf_{t > 0}\ \Expect \ntr \left[
	\frac{t}{2} \Expect\big[ (\mtx{X} - \mtx{X}')^2 \condl Z\big]
	\cdot \econst^{\theta \mtx{X}}  
	+ \frac{1}{2t} \Expect\big[ \mtx{K}(Z,Z')^2\condl Z\big] \cdot \econst^{\theta \mtx{X}}  \right]. %
\end{align*}
The first equality follows from the exchangeability of $(\mtx{X},\mtx{X}')$;
the second follows from the change of variables $s = \abs{\theta}^{-1} t$; 
and the final one depends on the pull-through property of conditional expectation.
We reach the result \eqref{eqn:m-prime-Delta} by introducing the definitions
\eqref{eqn:conditional-variance} and \eqref{eqn:K-conditional-variance} of the conditional variance and the kernel conditional variance.
\end{proof}

\subsection{Bounding the Trace Mgf} \label{sec:trace-mgf}

\lemref{mgf-derivative} gives us a powerful tool for bounding the trace mgf of a random matrix $\mtx{X}$ that is presented as part of a kernel Stein pair.  The following lemma shows how to derive a trace mgf bound from bounds on the kernel conditional variance.

\begin{lemma} [Trace Mgf Estimates for Bounded Random Matrices] \label{lem:mgf-bounds}
Let $(\mtx{X}, \mtx{X}')$ be a $\mtx{K}$-Stein pair, and suppose there exist nonnegative constants $c, v, s$ for which 
\begin{equation} \label{eqn:comparison2}
\mtx{V}_{\mtx{X}} \psdle s^{-1} (c \mtx{X} + v \, \Id)  
\quad\text{and}\quad
\mtx{V}^{\mtx{K}} \psdle s \, (c \mtx{X} + v \, \Id)
\quad \text{almost surely}.
\end{equation}
Then the normalized trace mgf $m(\theta) \defby \Expect \ntr \econst^{\theta \mtx{X}}$ satisfies the bounds
\begin{align}
\log m(\theta)
	&\leq \frac{v \theta^2}{2}  
	&&\text{when $\theta \leq 0$.}  \label{eqn:mgf-negative} \\
\log m(\theta)
	&\leq \frac{v}{c^2} \left[ \log\left( \frac{1}{1-c\theta}\right) - c\theta \right]
		\label{eqn:mgf-positive-1} \\
	&\leq \frac{v \theta^2}{ 2(1 - c \theta)}
	&&\text{when $0 \leq \theta < 1/c$.} \label{eqn:mgf-positive-2}
\end{align}
The two conditional variances are defined in
\eqref{eqn:conditional-variance} and \eqref{eqn:K-conditional-variance}.
\end{lemma}

\begin{proof}
As demonstrated in \cite[Lem.~4.3]{MackeyJoChFaTr12}, the assumption \eqref{eqn:comparison2} implies that $\mtx{X}$ is almost surely bounded in norm.
Hence, we may apply \lemref{mgf-derivative} along with our conditional variance bounds \eqref{eqn:comparison2} to obtain
\begin{align*}
\abs{m'(\theta)}
	&\leq  \frac{1}{2} \abs{\theta} \cdot \inf_{t > 0}\ \Expect \ntr\big[(t\,\mtx{V}_{\mtx{X}} + t^{-1} \mtx{V}^{\mtx{K}}) \, \econst^{\theta \mtx{X}} \big] \\
	&\leq  \frac{1}{2} \abs{\theta} \cdot \Expect \ntr\big[(s \, \mtx{V}_{\mtx{X}} + s^{-1} \mtx{V}^{\mtx{K}}) \, \econst^{\theta \mtx{X}} \big] \\	
	&\leq \abs{\theta} \cdot \Expect \ntr\big[ (c \mtx{X} + v \, \Id) \, \econst^{\theta \mtx{X}} \big] \\
	&= c \abs{\theta} \cdot \Expect \ntr\big[ \mtx{X} \econst^{\theta \mtx{X}} \big]
		+ v \abs{\theta} \cdot \Expect \ntr \econst^{\theta \mtx{X}} \\
	&= c \abs{\theta} \cdot m'(\theta) + v \abs{\theta} \cdot m(\theta),  
\end{align*}
where the third inequality follows from the positivity of $\econst^{\theta \mtx{X}}$.  
The remainder of the argument now proceeds as in \cite[Lem.~4.3]{MackeyJoChFaTr12}.
\end{proof}

\subsection{Proof of Theorem~\ref{thm:concentration-bdd}}

The remainder of the proof of \thmref{concentration-bdd} is identical to that of \cite[Thm.~4.1]{MackeyJoChFaTr12}, once we substitute the trace mgf estimates from Lemma~\ref{lem:mgf-bounds} in place of the result \cite[Lem.~4.3]{MackeyJoChFaTr12}.  We omit the details.

\section{Proof of the Polynomial Moment Inequality} \label{sec:proof-bdg-inequality}

Next, we develop a proof of the matrix Burkholder--Davis--Gundy inequality,
Theorem~\ref{thm:bdg-inequality}.  The proof parallels the argument
in~\cite{MackeyJoChFaTr12}, but we need some new matrix inequalities
to make the extension to kernel Stein pairs.

\subsection{The Polynomial Mean Value Trace Inequality}

The critical new ingredient in Theorem~\ref{thm:bdg-inequality}
is the polynomial mean value trace inequality, Lemma~\ref{lem:pmvti}.
Let us proceed with a proof of this result.

\begin{proof}[Proof of Lemma~\ref{lem:pmvti}] 
First, we need to develop another representation for the trace
quantity that we are analyzing.  Assume that $\mtx{A}, \mtx{B}, \mtx{C} \in \Sym{d}$.  A direct calculation shows that
\begin{align*}
\mtx{A}^{q} - \mtx{B}^{q} 
	= \sum\nolimits_{k=0}^{q-1}\mtx{A}^{k} (\mtx{A} - \mtx{B}) \mtx{B}^{q-1-k}.
\end{align*}
As a consequence,
\begin{equation} \label{eqn:poly-trace-sum}
\trace \left[\mtx{C} (\mtx{A}^{q} - \mtx{B}^{q} )\right] 
	= \sum\nolimits_{k=0}^{q-1} \ip{ \mtx{C} }{ \mtx{A}^{k} (\mtx{A} - \mtx{B}) \mtx{B}^{q-1-k} }.
\end{equation}
To bound the right-hand side of~\eqref{eqn:poly-trace-sum}, we require an approriate mean inequality.

To that end, we define some self-adjoint operators on $\M^d$:
$$
\mathcal{A}_k(\mtx{M}) \defby \mtx{A}^k \mtx{M}
\qtext{and}
\mathcal{B}_{k}(\mtx{M}) \defby  \mtx{M} \mtx{B}^{k}
\quad\text{for each $k = 0, 1, 2, \dots, q-1$.} 
$$
The absolute values of these operators satisfy
$$
\abs{\mathcal{A}_k}(\mtx{M}) = \abs{\mtx{A}}^k \mtx{M}
\qtext{and}
\abs{\mathcal{B}_{k}}(\mtx{M}) =  \mtx{M} \abs{\mtx{B}}^{k}
\quad\text{for each $k = 0, 1, 2, \dots, q-1$.} 
$$
Note that $\abs{\mathcal{A}_k}$ and $\abs{\mathcal{B}_{q-k-1}}$ commute with each other
for each $k$.  Therefore, Young's inequality for commuting operators, Lemma~\ref{lem:young-commute}, yields the bound
\begin{align} 
\abs{\mathcal{A}_k\mathcal{B}_{q-k-1}}
	= \abs{\mathcal{A}_k} \abs{\mathcal{B}_{q-k-1}}
	&\psdle \frac{k}{q-1}\abs{\mathcal{A}_k}^{(q-1)/k}
	+ \frac{q-k-1}{q-1} \abs{\mathcal{B}_{q-k-1}}^{(q-1)/(q-k-1)} \notag \\
	&= \frac{k}{q-1} \abs{\mathcal{A}_1}^{q-1} + \frac{q-k-1}{q-1}\abs{\mathcal{B}_1}^{q-1}. \label{eqn:poly-mean-ineq}
\end{align}
Summing over $k$, we discover that
\begin{align}\label{eqn:young-bound}
\sum\nolimits_{k=0}^{q-1}\abs{\mathcal{A}_k\mathcal{B}_{q-k-1} }
	&\psdle \frac{q}{2} \abs{\mathcal{A}_1}^{q-1} + \frac{q}{2}\abs{\mathcal{B}_1}^{q-1}.
\end{align}
This is the mean inequality that we require.

To apply this result, we need to rewrite~\eqref{eqn:poly-trace-sum} using the operators $\mathcal{A}_k$ and $\mathcal{A}_{q-k-1}$.  It holds that
\begin{multline} \label{eqn:poly-sums}
\trace \left[\mtx{C} (\mtx{A}^{q} - \mtx{B}^{q} )\right]
	= \sum_{k=0}^{q-1} \ip{ \mtx{C} }{ (\mathcal{A}_k \mathcal{B}_{q-k-1})(\mtx{A}-\mtx{B}) } \\
	\leq \left[ \sum_{k=0}^{q-1} \ip{ \mtx{C} }{ \abs{\mathcal{A}_k \mathcal{B}_{q-k-1}}(\mtx{C}) } \cdot \sum_{k=0}^{q-1} \ip{ \mtx{A} - \mtx{B} }{  \abs{\mathcal{A}_k \mathcal{B}_{q-k-1}}(\mtx{A}-\mtx{B}) } \right]^{1/2}. 
\end{multline}
The second relation follows from the operator Cauchy--Schwarz inequality, Lemma~\ref{lem:operator-cs}, and the usual Cauchy--Schwarz inequality for the sum.

It remains to bound to two sums on the right-hand side of~\eqref{eqn:poly-sums}. 
The mean inequality~\eqref{eqn:poly-mean-ineq} ensures that
\begin{multline} \label{eqn:poly-sum-1}
\sum_{k=0}^{q-1} \ip{ \mtx{C} }{ \abs{\mathcal{A}_k \mathcal{B}_{q-k-1}}(\mtx{C}) }
	\leq \frac{q}{2} \ip{ \mtx{C} }{ \big(\abs{\mathcal{A}_1}^{q-1} + \abs{\mathcal{B}_1}^{q-1} \big)(\mtx{C}) } \\
	= \frac{q}{2} \ip{ \mtx{C} }{ \abs{\mtx{A}}^{q-1} \mtx{C} + \mtx{C} \abs{\mtx{B}}^{q-1} }
	= \frac{q}{2} \trace\big[ \mtx{C}^2 \big( \abs{\mtx{A}}^{q-1} + \abs{\mtx{B}}^{q-1} \big) \big].
\end{multline}
Likewise,
\begin{equation} \label{eqn:poly-sum-2}
\sum_{k=0}^{q-1} \ip{ \mtx{A} - \mtx{B} }{ \abs{\mathcal{A}_k \mathcal{B}_{q-k-1}}(\mtx{A}-\mtx{B}) }
	\leq \frac{q}{2} \trace\big[ (\mtx{A}-\mtx{B})^2 \big( \abs{\mtx{A}}^{q-1} + \abs{\mtx{B}}^{q-1} \big) \big].
\end{equation}
Introduce the two inequalities~\eqref{eqn:poly-sum-1} and~\eqref{eqn:poly-sum-2} into~\eqref{eqn:poly-sums} to reach
$$
\trace \left[\mtx{C} (\mtx{A}^{q} - \mtx{B}^{q} )\right]
	\leq \frac{q}{2} \bigg( \trace\big[ \mtx{C}^2 \big( \abs{\mtx{A}}^{q-1} + \abs{\mtx{B}}^{q-1} \big) \big] \cdot
	\trace\big[ (\mtx{A}-\mtx{B})^2 \big( \abs{\mtx{A}}^{q-1} + \abs{\mtx{B}}^{q-1} \big) \big] \bigg)^{1/2}.
$$
The result follows when we apply the numerical inequality between the geometric mean and the arithmetic mean.
\end{proof}

\subsection{Proof of Theorem~\ref{thm:bdg-inequality}}

Abbreviate
\begin{align*}
E \defby \Expect \pnorm{2p}{\mtx{X}}^{2p} = \Expect \trace \abs{\mtx{X}}^{2p} =  \Expect \trace \big[ \mtx{X}\cdot{\mtx{X}}^{2p-1}\big].
\end{align*}
To apply the method of exchangeable pairs, Lemma~\ref{lem:exchange}, we check the regularity condition~\eqref{eqn:regularity-mep}:
\begin{align*} 
\Expect \smnorm{}{ \mtx{K}(Z,Z') \cdot {\mtx{X}}^{2p-1} }
	&\leq \Expect \big( \norm{ \mtx{K}(Z,Z')}  \norm{ \mtx{X} }^{2p-1} \big) \\
	&\leq \big(\Expect \norm{ \mtx{K}(Z,Z')}^{2p} \big)^{1/2p} \big(\Expect \norm{ \mtx{X} }^{2p} \big)^{(2p-1)/2p} 
	< \infty,
\end{align*}
where we have applied \Holder's inequality for expectation and the fact that the spectral norm is dominated by the Schatten $(2p)$-norm.
Invoke Lemma~\ref{lem:exchange} with $\mtx{F}(\mtx{X}) = {\mtx{X}}^{2p - 1}$ to reach
\begin{align*}
E = \frac{1}{2} \Expect \trace\big[ \mtx{K}(Z,Z') \cdot
	\big( {\mtx{X}}^{2p-1} - {\mtx{X}'}^{2p-1} \big) \big].
\end{align*}
Next, fix a parameter $s > 0$.  Apply the polynomial mean value trace inequality, Lemma~\ref{lem:pmvti}, with $q=2p-1$ to obtain the estimate
\begin{align*}
E 
	&\leq \frac{2p-1}{8} \Expect \trace[(s \, (\mtx{X}-\mtx{X}')^2+s^{-1} \mtx{K}(Z,Z')^2) \cdot (\mtx{X}^{2p-2} + \mtx{X}'^{2p-2})] \\
	&= \frac{2p-1}{4} \Expect \trace[(s \, (\mtx{X}-\mtx{X}')^2+ s^{-1} \mtx{K}(Z,Z')^2)\cdot\mtx{X}^{2p-2}] \\
	&= (2p-1) \Expect \trace\left[\frac{1}{2}(s \, \mtx{V}_{\mtx{X}}+s^{-1} \mtx{V}^{\mtx{K}})\cdot \mtx{X}^{2p-2}\right],
\end{align*}
where we have used the exchangeability of $(\mtx{X},\mtx{X}')$
and the definitions~\eqref{eqn:conditional-variance} and \eqref{eqn:K-conditional-variance} of the conditional variances.
In the last step, we justify the pull-through property with the regularity condition $\Expect \pnorm{2p}{\mtx{X}}^{2p} < \infty$.
The remainder of the argument is identical with the proof of~\cite[Thm.~8.1]{MackeyJoChFaTr12}.

\section{Haar Measures and Controlled Total Variation} \label{sec:haar}
In this section, we prove \cororef{concentration-haar} by studying the behavior of Hermitian functions of group-valued random elements.
Under the notation of \cororef{concentration-haar}, we define $\mtx{X} \defby \mtx{\Psi}(Z)$.
\cite[Thm.~4.6]{Cha08:Concentration-Inequalities} showed that scalar functions of the Haar measure are well concentrated whenever particular random walks on $G$ converge rapidly to the Haar distribution.
In the sections to follow, we develop a Hermitian analogue of this relationship using the tools of \secsref{exchange}~and \secssref{concentration-bdd}.
As in \cite{Cha08:Concentration-Inequalities}, we will adopt the total variation distance between measures~\eqref{eqn:tv} as our convergence metric.

\subsection{A Kernel Coupling} \label{sec:haar-coupling}
We begin by establishing a kernel coupling suitable for analyzing $\mtx{X}$.
Since $Y\in G$ is independent of $Z$ and satisfies \eqref{eqn:Y-property}, 
$Z' = YZ$ is exchangeable counterpart for $Z$, and hence $(\mtx{X},\mtx{X}') = (\mtx{\Psi}(Z),\mtx{\Psi}(Z'))$ is an exchangeable pair.

Moreover, the sequence of pairs
\begin{align} \label{eqn:haar-coupling}
(Z_{(i)},Z_{(i)}') \defby (Y_i\cdots Y_1Z, Y_i\cdots Y_1Z') \qtext{for each } i\geq 0
\end{align} 
defines a kernel coupling for $(\mtx{X},\mtx{X}')$.
Thus, $(\mtx{X},\mtx{X}')$ is a kernel Stein pair with $\mtx{K}$ defined as in Lemma~\ref{lem:kernel-coupling}
whenever the precondition \eqref{eqn:coupling-premise} is met.

\subsection{The Conditional Variances}
The sequence of multipliers $(Y_i\cdots Y_1)_{i=1}^\infty$ in our kernel coupling~\eqref{eqn:haar-coupling} 
can be viewed as a random walk on the group $G$,
and, for many choices of $Y$, this sequence will converge to a Haar distributed random variable.
Intuitively, a faster rate of convergence implies a faster coupling time for the Markov chains $(Z_{(i)})_{i\geq 0}$ and $(Z'_{(i)})_{i\geq 0}$ and hence a smaller $\mtx{K}$-conditional variance \eqref{eqn:K-conditional-variance}.
Our next lemma makes this intuition more precise by bounding the $\mtx{K}$-conditional variance in terms of the 
total variation distance between $Y_i\cdots Y_1$ and $Z$.

\begin{lemma} \label{lem:total-variation}
Let $Z\sim \mu$ be Haar distributed on a group $G$.  Let $(\mtx{X},\mtx{X}') \defby (\mtx{\Psi}(Z),\mtx{\Psi}(Z'))$ 
with $\mtx{K}$ constructed as in \secref{haar-coupling}.
Suppose that $\mu_i$ is the distribution of ${Y}_i\cdots {Y}_1$ and that
\[
S^2 \defby \sup_{g \in G}\norm{\Expect[(\mtx{\Psi}(g) - \mtx{\Psi}(Yg))^2]}.
\]
Then $(\mtx{X},\mtx{X}')$ is a $\mtx{K}$-Stein pair whenever $\sum\nolimits_{i=0}^\infty\tv(\mu_i, \mu) < \infty$.
Moreover, the conditional variance \eqref{eqn:conditional-variance} satisfies
\begin{align*} 
\maxeig{\mtx{V}_{\mtx{X}} }
	\le \frac{S^2}{2} \quad\text{almost surely},
\end{align*}
and the $\mtx{K}$-conditional variance \eqref{eqn:K-conditional-variance} satisfies
\begin{align*} 
\maxeig{\mtx{V}^{\mtx{K}} }
	\le \frac{S^2}{2}\left(\sum\nolimits_{i=0}^\infty \min\big\{1,\ 4RS^{-1}\, \tv(\mu_i, \mu)\big\}\right)^2\quad\text{almost surely}.
\end{align*}
\end{lemma}

\begin{proof}
Fix any $i \geq 0$.
We aim to bound
\begin{align*}
\Expect[\mtx{\Psi}(Z_{(i)}) - \mtx{\Psi}(Z_{(i)}')\condl Z=z,Z'=z']^2
	&= (\Expect[\mtx{\Psi}(Y_i\cdots Y_1 z)] - \Expect[\mtx{\Psi}(Y_i\cdots Y_1 z')])^2 \\
	&\psdle 2 \Expect[\mtx{\Psi}(Y_i\cdots Y_1 z)]^2 +2\Expect[\mtx{\Psi}(Y_i\cdots Y_1 z')]^2,
\end{align*}
where the inequality follows from the convexity of the matrix square.
For any $z\in G$,
\begin{align*}
	{\Expect[\mtx{\Psi}(Y_i\cdots Y_1 z)] }
		&= {\Expect[\mtx{\Psi}(Y_i\cdots Y_1 z)] - \Expect[\mtx{\Psi}(Z)]} \\
		&= {\Expect[\mtx{\Psi}(Y_i\cdots Y_1 z)] - \Expect[\mtx{\Psi}(Zz)]},
\end{align*}
since $Z$ is Haar distributed, and hence $Zz =_d Z$.
Furthermore, for any positive measure $\nu$ that dominates $\mu$ and $\mu_i$,
\begin{align*}
	\norm{\Expect[\mtx{\Psi}(Y_i\cdots Y_1 z)] }
		&= \norm{\int \mtx{\Psi}(yz)\left(\frac{d\mu_i}{d\nu}(y) - \frac{d\mu}{d\nu}(y)\right)d\nu(y)} \\
		&\leq R\int \left|\frac{d\mu_i}{d\nu}(y) - \frac{d\mu}{d\nu}(y)\right|d\nu(y) \\
		&\leq 2R\ \tv(\mu_i, \mu),
\end{align*}
by our bound on $\mtx{\Psi}$ and the definition of total variation. 
Therefore,
\begin{align*}
\Expect[\Expect[\mtx{\Psi}(Z_{(i)}) - \mtx{\Psi}(Z_{(i)}')\condl Z,Z']^2\condl Z]
	&\psdle 16R^2 \tv^2(\mu_i, \mu)\ \Id.
\end{align*}
We note moreover that
\[
\norm{\sum\nolimits_{i=0}^\infty |\Expect[\mtx{\Psi}(Z_{(i)}) - \mtx{\Psi}(Z'_{(i)})\condl Z_{(0)}=z,Z'_{(0)}=z']|}
	\leq \sum\nolimits_{i=0}^\infty 4R\ \tv(\mu_i, \mu)
\]
for all $z$ and $z'$.  Hence, by \lemref{kernel-coupling}, $(\mtx{X},\mtx{X}')$ is a valid $\mtx{K}$-Stein pair 
whenever the total variation distances are summable.

Next, let $W_i \defby Y_i\cdots Y_1$, and notice that
\begin{align*}
\norm{\Expect[\Expect[\mtx{\Psi}(Z_{(i)}) - \mtx{\Psi}(Z_{(i)}')\condl Z,Z']^2\condl Z = z]}
	&\le \norm{\Expect[(\mtx{\Psi}(Z_{(i)}) - \mtx{\Psi}(Z_{(i)}'))^2\condl Z = z]} \\
	&= \norm{\Expect[(\mtx{\Psi}(W_iz) - \mtx{\Psi}(W_iYz))^2] } \\
	&\leq \sup_{g\in G}\norm{\Expect[(\mtx{\Psi}(gz) - \mtx{\Psi}(gYz))^2] } \\
	&= \sup_{g\in G}\norm{\Expect[(\mtx{\Psi}(gz) - \mtx{\Psi}(Ygz))^2] } \leq S^2,
\end{align*}
where the first inequality is a consequence of the convexity of the matrix square,
and the final equality follows from the property $g^{-1}Yg =_d Y$ for all $g\in G$.
Hence, we may apply Lemma~\ref{lem:conditional-variance-bound} with 
$s_0 = S$, with $s_i = \min\big\{S,\ 4R\, \tv(\mu_i, \mu)\big\}$ for $i > 0$, and with $\mtx{\Gamma}(Z) = \Id$
to obtain the result.
\end{proof}

\subsection{Exponential Concentration}
We are finally equipped to prove \cororef{concentration-haar}.
Under the kernel coupling construction of \secref{haar-coupling},
the conditional variance bounds of \lemref{total-variation} 
imply that \thmref{concentration-bdd} holds with $c=0$,
with $v = \tfrac{1}{2}S^2\sum\nolimits_{i=0}^\infty \min\big\{1,\,4RS^{-1}\, \tv(\mu_i, \mu)\big\}$, 
and with $s =\sum\nolimits_{i=0}^\infty \min\big\{1, \ 4RS^{-1} \, \tv(\mu_i, \mu)\big\}$.
This establishes the result.

\bibliographystyle{amsplain}
\bibliography{refs-stein_kernel.bib}

\end{document}